\documentclass[10pt, a4paper, twoside]{amsart}
\usepackage[usenames, dvipsnames]{color}
\definecolor{darkblue}{rgb}{0.0, 0.0, 0.45}

\usepackage[colorlinks	= true,
			raiselinks	= true,
			linkcolor	= darkblue, %MidnightBlue,
			citecolor	= Mahogany,
			urlcolor	= ForestGreen,
			pdfauthor	= {Peyman Mohajerin Esfahani},
			pdftitle	= {},
			pdfkeywords	= {},
			pdfsubject	= {},
			plainpages	= false]{hyperref}

\usepackage{dsfont,amssymb,amsmath,subfigure, graphicx} %,enumitem} %,etaremune,savesym}
\usepackage{amsfonts,dsfont,mathtools, mathrsfs,amsthm,amsxtra} 
\usepackage[amssymb, thickqspace]{SIunits}
\usepackage{algorithm,algorithmicx}

\usepackage{enumitem}%,etaremune,savesym}

\allowdisplaybreaks
\date{\today}
\renewcommand\thesection{\arabic{section}}

\newtheorem{Thm}{Theorem} [section]
\newtheorem{Prop}[Thm]{Proposition}
\newtheorem{Fact}[Thm]{Fact}
\newtheorem{Lem}[Thm]{Lemma}
\newtheorem{Cor}[Thm]{Corollary}

\newtheorem{As}[Thm]{Assumption}

\newtheorem{Def}[Thm]{Definition}

\newtheorem{Rem}[Thm]{Remark}

\newcommand{\PP}{\mathds{P}}
\newcommand{\EE}{\mathds{E}}
\newcommand{\R}{\mathbb{R}}
\newcommand{\N}{\mathbb{N}}
\newcommand{\eps}{\varepsilon}
\DeclareMathOperator{\e}{e}

\newcommand{\ra}{\rightarrow}
\newcommand{\da}{\downarrow}

\newcommand{\ind}[1]{\mathds{1}_{#1}}
\newcommand{\Let}{\coloneqq}
\newcommand{\teL}{\eqqcolon}
\newcommand{\diff}{\mathrm{d}}

\newcommand{\tr}{^\intercal}
\newcommand{\opt}{^\star}

\newcommand{\sigalg}{\mathcal{F}}

\DeclareMathOperator{\Ap}{AP}
\DeclareMathOperator{\Cp}{CP}
\DeclareMathOperator{\Rp}{RP}
\DeclareMathOperator{\APr}{\widetilde{AP}}
\DeclareMathOperator{\CPr}{\widetilde{CP}}

\newcommand{\M}{\mathcal{M}}
\newcommand{\NN}{\mathcal{N}}
\newcommand{\B}{\mathcal{B}}
\newcommand{\W}{\mathcal{W}}

\newcommand{\Lnorm}{\mathcal{L}_2}
\newcommand{\Hinf}{\mathcal{H}_\infty}
\newcommand{\proj}{\mathds{T}_{\B}}
\newcommand{\inner}[1]{\left \langle #1 \right \rangle}

\DeclareMathOperator{\CP}{CP}

\newcommand{\D}{\mathcal{D}}

\newcommand{\filter}{\mathfrak{F}}

\graphicspath{{fig/}}

\usepackage{fancyhdr}
\title{A Tractable Fault Detection and Isolation Approach for Nonlinear Systems with Probabilistic Performance}
\thanks{The authors are with the Automatic Control Laboratory, ETH Z\"urich, 8092 Z\"urich, Switzerland. Emails: \texttt{\{mohajerin,lygeros\}@control.ee.ethz.ch}}

\author{Peyman Mohajerin Esfahani and John Lygeros}

\begin{document}

\begin{abstract}
This article presents a novel perspective along with a scalable methodology to design a fault detection and isolation (FDI) filter for high dimensional nonlinear systems. Previous approaches on FDI problems are either confined to linear systems or they are only applicable to low dimensional dynamics with specific structures. In contrast, shifting attention from the system dynamics to the disturbance inputs, we propose a relaxed design perspective to train a linear residual generator given some statistical information about the disturbance patterns. That is, we propose an optimization-based approach to robustify the filter with respect to finitely many signatures of the nonlinearity. We then invoke recent results in randomized optimization to provide theoretical guarantees for the performance of the proposed filer. Finally, motivated by a cyber-physical attack emanating from the vulnerabilities introduced by the interaction between IT infrastructure and power system, we deploy the developed theoretical results to detect such an intrusion before the functionality of the power system is disrupted.
\end{abstract}
\maketitle

%==========================================================================================
\section{Introduction} \label{FDI:sec:int}
%==========================================================================================
	The task of FDI in control systems involves generating a diagnostic signal sensitive to the occurrence of specific faults. This task is typically accomplished by designing a filter with all available information as inputs (e.g., control signals and given measurements) and a scalar output that implements a non-zero mapping from the fault to the diagnostic signal, which is known as the residual, while decoupling unknown disturbances. The concept of residual plays a central role for the FDI problem which has been extensively studied in the last two decades. 
	
	In the context of linear systems, Beard and Jones \cite{ref:Beard_PhD,ref:Jones_PhD} pioneered an observer-based approach whose intrinsic limitation was later improved by Massoumnia et al. \cite{ref:MassoumniaWillsky}. Following the same principles but from a game theoretic perspective, Speyer and coauthors thoroughly investigated the approach in the presence of noisy measurements \cite{ref:ChungSpeyer_98,ref:DougSpeyer_99}. Nyberg and Frisk extended the class of systems to linear differential-algebraic equation (DAE) apparently subsuming all the previous linear classes \cite{ref:NybergFrisk_TAC06}, which recently also studied in the context of stochastic linear systems \cite{ref:EriFriKry-13}. This extension greatly enhanced the applicability of FDI methods since the DAE models appear in a wide range of applications, including electrical systems, robotic manipulators, and mechanical systems. 
	
	For nonlinear systems, a natural approach is to linearize the model at an operating point, treat the nonlinear higher order terms as disturbances, and decouple their contributions from the residual by employing robust techniques  \cite{ref:SeligerFrank91,ref:HouPatton_96}. This strategy only works well if either the system remains close to the chosen operating point, or the exact decoupling is possible. The former approach is often limited, since in the presence of unknown inputs the system may have a wide dynamic operating range, which in case linearization leads to a large mismatch between linear model and nonlinear behavior. The latter approach was explored in detail by De Persis and Isidori, who in \cite{ref:DePersisIsidori} proposed a differential geometric approach to extend the unobservibility subspaces of \cite[Section IV]{ref:Massoumnia}, and by Chen and Patton, who in \cite[Section 9.2]{ref:Chen} dealt with a particular class of bilinear systems. These methods are, however, practically limited by the need to verify the required conditions on the system dynamics and transfer them into a standard form, which essentially involve solving partial differential equations, restricting the application of the method to relatively low dimensional systems.
	
	Motivated by this shortcoming, in this article we develop a novel approach to FDI which strikes a balance between analytical and computational tractability, and is applicable to high dimensional nonlinear dynamics. For this purpose, we propose a design perspective that basically shifts the emphasis from the system dynamics to the family of disturbances that the system may encounter. We assume that some statistical information of the disturbance patterns is available. Following \cite{ref:NybergFrisk_TAC06} we restrict the FDI filters to a class of linear operators that fully decouple the contribution of the linear part of the dynamics. Thanks to the linearity of the resulting filter, we then trace the contribution of the nonlinear term to the residual, and propose an optimization-based methodology to robustify the filter to the nonlinearity signatures of the dynamics by exploiting the statistical properties of the disturbance signals. The optimization formulation is effectively convex and hence tractable for high dimensional dynamics. Some preliminary results in this direction were reported in \cite{ref:Mohajerin_CDC12}, while an application of our approach in the presence of measurement noise was successfully tested for wind turbines in \cite{ref:SveMohKamLyg-13}.
	
	The performance of the proposed methodology is illustrated in an application to an emerging problem of cyber security in power networks. In modern power systems, the cyber-physical interaction of IT infrastructure (SCADA systems) with physical power systems renders the system vulnerable not only to operational errors but also to malicious external intrusions. As an example of this type of cyber-physical interaction we consider here the Automatic Generation Control (AGC) system, which is one of the few control loops in power networks that are closed over the SCADA system without human operator intervention. In earlier work \cite{ref:Mohajerin_ACC10,ref:Mohajerin_CDC10} we have shown that, having gained access to the AGC signal, an attacker can provoke frequency deviations and power oscillations by applying sophisticated attack signals. The resulting disruption can be serious enough to trigger generator out-of-step protection relays, leading to load shedding and generator tripping. Our earlier work, however, also indicated that an early detection of the intrusion may allow one to disconnect the AGC and limit the damage by relying solely on the so-called primary frequency controllers. In this work we show how to mitigate this cyber-physical security concern by using the proposed FDI scheme to develop a protection layer which quickly detects the abnormal signals generated by the attacker. This approach to enhancing the cyber-security of power transmission systems led to an EU patent sponsored by ETH Zurich \cite{ref:Patent}.
	
	The article is organized as follows. In Section \ref{FDI:sec:prob} a formal description of the FDI problem as well as the outline of the proposed methodology is presented. A general class of nonlinear models is described in Section \ref{FDI:sec:model}. Then, reviewing residual generation for the linear models, we develop an optimization-based framework for nonlinear systems in Section \ref{FDI:sec:FDI}. Theoretical guarantees are also provided in the context of randomized algorithms. We apply the developed methodology to the AGC case study in Section \ref{FDI:sec:sim}, and finally conclude with some remarks and directions for future work in Section \ref{FDI:sec:sum}. For better readability, the technical proofs of Sections \ref{FDI:subsec:nonlinear} and \ref{FDI:subsec:perf} are moved to the appendices.

%===============================================================================
\paragraph{\bf Notation}
%===============================================================================
The symbols $\N$ and $\R_+$ denote the set of natural and nonnegative real numbers, respectively. Let $A \in \R^{n \times m}$ be an $n \times m$ matrix with real values, $A\tr \in \R^{m \times n}$ be its transpose, and $\|A\|_2 \Let \overline \sigma(A)$ where  $\overline{\sigma}$ is the maximum singular value of the matrix.  Given a vector $v \Let [v_1, \cdots, v_n]\tr$, the infinite norm is defined as $\| v \|_\infty \Let \max_{i\le n} |v_i|$. Let $G$ be a linear matrix transfer function. Then $\|G\|_{\Hinf} \Let \sup_{\omega \in \R} \overline{\sigma}\big(G(j\omega)\big)$, where $\overline{\sigma}$ is the maximum singular value of the matrix $G(j\omega)$. The function space $\W^n$ denotes the set of piece-wise continuous (p.w.c) functions taking values in $\R^n$, and $\W^n_T$ is the restriction of $\W^n$ to the time interval $[0,T]$, which is endowed with the $\Lnorm$-inner product, i.e., $\inner{e_1,e_2} \Let \int_0^T e_1\tr(t)e_2(t) \diff t$ with the associated $\Lnorm$-norm $\|e\|_{\Lnorm} \Let \sqrt{\inner{e,e}}$. The linear operator $p:\W^n \ra \W^n$ is the distributional derivative operator. In particular, if $e:\R_+ \ra \R^n$ is a smooth mapping then $p [e(t)] \Let \frac{\diff}{\diff t}e(t)$. Given a  probability space $(\Omega, \sigalg, \PP)$, we denote the $n$-Cartesian product space by $\Omega^n \Let \bigotimes_{i=1}^n \Omega$ and the respective product measure by $\PP^n$.
\section{Problem Statement and Outline of the Proposed Approach} \label{FDI:sec:prob}
%==========================================================================================
In this section, we provide the formal description of the FDI problem as well as our new design perspective. We will also outline our methodology to tackle the proposed perspective.

	%------------------------------------------------------------------------------------------
	\subsection{Formal Description}
	%------------------------------------------------------------------------------------------
	The objective of the FDI design is to use all information to generate a diagnostic signal to alert the operators to the occurrence of a specific fault. Consider a general dynamical system as in Figure \ref{fig:sys} with its inputs categorized into (i) unknown inputs $d$, (ii) fault signal $f$, and (iii) known inputs $u$. The unknown input $d$ represents unknown disturbances that the dynamical system encounters during normal operation. The known input $u$ contains all known signals injected to the system which together with the measurements $y$ are available for FDI tasks. Finally, the input $f$ is a fault (or an intrusion) which cannot be directly measured and represents the signal to be detected.
	
		\begin{figure}[t!]
			\centering
			\includegraphics[scale = 0.4]{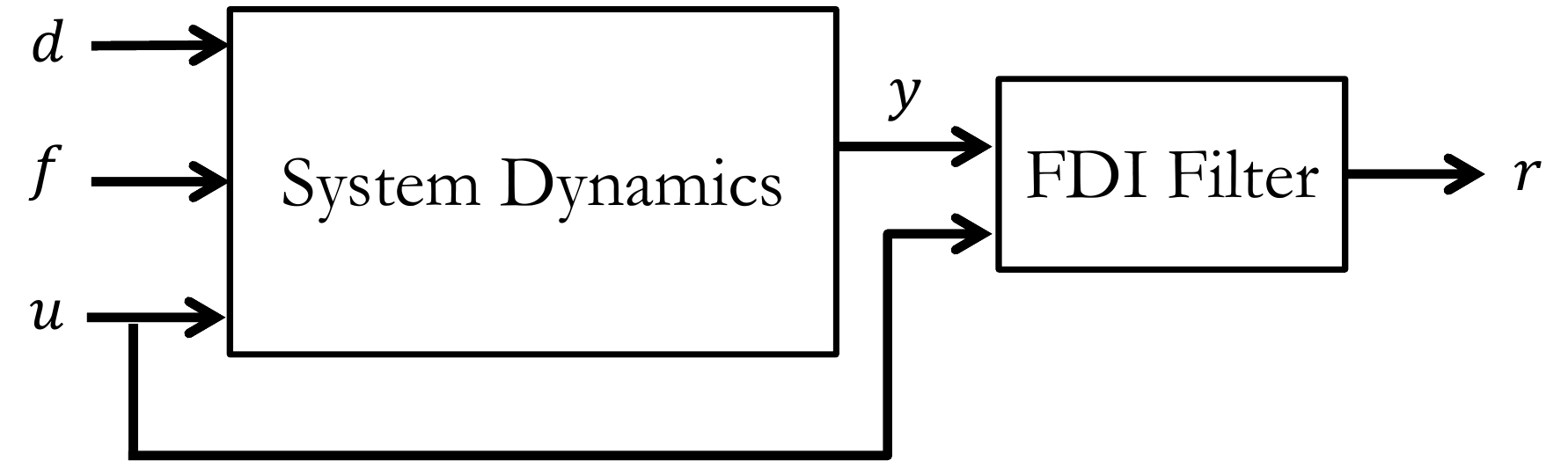}
			\caption{General configuration of the FDI filter}
			\label{fig:sys}
		\end{figure}
	
	The FDI task is to design a filter whose input are the known signals ($u$ and $y$) and whose output (known as the residual and denoted by $r$) differentiates whether the measurements are a consequence of some normal disturbance input $d$, or due to the fault signal $f$. Formally speaking, the residual can be viewed as a function $r(d,f)$, and the FDI design is ideally translated as the mapping requirements
		\begin{subequations}
			\label{mapping}
				\begin{align}
				\label{d-r} d &\mapsto r(d,0)	\equiv 0, \\
				\label{f-r} f &\mapsto r(d,f)	\neq 0, \quad \forall d
				\end{align}
		\end{subequations}
	where condition \eqref{d-r} ensures that the residual of the filter, $r$, is not excited when the system is perturbed by normal disturbances $d$, while condition \eqref{f-r} guarantees the filter sensitivity to the fault $f$ in the presence of any disturbance $d$. 
	
	The state of the art in FDI concentrates on the system dynamics, and imposes restrictions to provide theoretical guarantees for the required mapping conditions \eqref{mapping}. For example, the authors in \cite{ref:NybergFrisk_TAC06} restrict the system to linear dynamics, whereas \cite{ref:Hamouri_99,ref:DePersisIsidori} treat nonlinear systems but impose necessary conditions in terms of a certain distribution connected to their dynamics. In an attempt to relax the perfect decoupling condition, one may consider the worst case scenario of the mapping \eqref{mapping} in a robust formulation as 
		\begin{align} 
		\label{rcp}
			\Rp: \left\{ \begin{array}{lll}
				\min\limits_{\gamma, \filter}		&  \gamma   \\
				\text{s.t. }	& \big \|r(d,0) \big \| \leq \gamma, &\forall d \in \D \vspace{1mm} \\
								& f \mapsto r(d,f) \neq 0,     &\forall d \in \D,
						\end{array} \right.
		\end{align}
	where $\D$ is set of normal disturbances, $\gamma$ is the alarm threshold of the designed filter, and the minimization is running over a given class of FDI filters denoted by $\filter$. Note that the residual $r$ is influenced by the choice of the filter in $\filter$, but we omit this dependence for notational simplicity. In view of formulation \eqref{rcp}, an alarm is only raised whenever the residual exceeds $\gamma$, i.e., the filter avoids any false alarm. This, however, comes at the cost of missed detections of the faults whose residual is not bigger than the threshold $\gamma$. In the literature, the robust perspective $\Rp$ has also been studied in order for a trade-off between disturbance rejection and fault sensitivity for a certain class of dynamics, e.g., see \cite[Section 9.2]{ref:Chen} for bilinear dynamics and \cite{ref:Franze-RobustPoly-12} for multivariate polynomial systems.

	%------------------------------------------------------------------------------------------
	\subsection{New Design Perspective}
	%------------------------------------------------------------------------------------------
	Here we shift our attention from the system dynamics to the class of unknown inputs $\D$. We assume that the disturbance signal $d$ comes from a prescribed probability space and relax the robust formulation $\Rp$ by introducing probabilistic constraints instead. In this view, the performance of the FDI filter is characterized in a probabilistic fashion. 
	
	Assume that the signal $d$ is modeled as a random variable on the prescribed probability space $(\Omega, \sigalg, \PP)$, which takes values in a metric space endowed with the corresponding Borel sigma-algebra. Assume further that the class of FDI filters ensures the measurability of the mapping $d \mapsto r$ where $r$ also belongs to a metric space. In light of this probabilistic framework, one may quantify the filter performance from different perspectives; in the following we propose two of them:
	{ %\small 
		\begin{align} 
		\label{rcp-ccp}
			\Ap: \left\{ \begin{array}{lll}
								\min\limits_{\gamma, \filter}		&  \gamma   \\
								\text{s.t. }	& \EE \big[ J \big( \|r(d,0)\| \big) \big ] \leq \gamma \vspace{1mm} \\
												& f \mapsto r(d,f) \neq 0, \quad \forall d \in \D,
						\end{array} \right.
			%\quad  
			\Cp: \left\{ \begin{array}{lll}
								\min\limits_{\gamma, \filter}		&  \gamma   \\
								\text{s.t. }	& \PP \big( \|r(d,0)\| \le \gamma \big ) \ge 1 - \eps \vspace{1mm} \\ 
												& f \mapsto r(d,f) \neq 0,  \quad \forall d \in \D,
						\end{array} \right.
		\end{align}}
	where $\EE[\cdot]$ in $\Ap$ is meant with respect to the probability measure $\PP$, and $\|\cdot\|$ is the corresponding norm in the $r$ space. The function $J:\R_+ \ra \R_+$ in $\Ap$ and $\eps \in (0,1)$ in $\Cp$ are design parameters. To control the filter residual generated by $d$, the payoff function $J$ is required to be in class $\mathcal{K}_\infty$, i.e., $J$ is strictly increasing and $J(0) = 0$ \cite[Definition 4.2, p.\ 144]{ref:Khalil}. The decision variables in the above optimization programs are $\filter$, a class of FDI filters which is chosen a priori, and $\gamma$ which is the filter threshold; we shall explain these design parameters more explicitly in subsequent sections. 

	Two formulations provide different probabilistic interpretations of fault detection. The program $\Ap$ stands for ``\emph{Average Performance}" and takes all possible disturbances into account, but in accordance with their occurrence probability in an averaging sense. The program $\CP$ stands for ``\emph{Chance Performance}" and ignores an $\eps$-fraction of the disturbance patterns and only aims to optimize the performance over the rest of the disturbance space. Note that in the $\CP$ perspective, the parameter $\eps$ is an additional design parameter to be chosen a priori.
	
	Let us highlight that the proposed perspectives rely on the probability distribution $\PP$, which requires prior information about possible disturbance patterns. That is, unlike the existing literature, the proposed design prioritizes between disturbance patterns in terms of their occurrence likelihood. From a practical point of view this requirement may be natural; in Section \ref{FDI:sec:sim} we will describe an application of this nature. 

	%------------------------------------------------------------------------------------------
	\subsection{Outline of the Proposed Methodology}
	%------------------------------------------------------------------------------------------
	We employ randomized algorithms to tackle the formulations in \eqref{rcp-ccp}. We generate $n$ independent and identically distributed (i.i.d.) scenarios $(d_i)_{i=1}^n$ from the probability space $(\Omega, \sigalg, \PP)$, and consider the following optimization problems as random counterparts of those in \eqref{rcp-ccp}:
		{ %\small
		\begin{align} 
		\label{prob-rand}
			\APr: \left\{ \begin{array}{lll}
								\min\limits_{\gamma, \filter}		&  \gamma   \\
								\text{s.t. }	& \frac{1}{n}\sum_{i = 1}^n J\big(\| r(d_i,0) \|\big)  \leq \gamma \vspace{1mm} \\
												& f \mapsto r(d,f) \neq 0, \quad \forall d \in \D
						\end{array} \right.
			\quad 
			\CPr: \left\{ \begin{array}{lll}
								\min\limits_{\gamma, \filter}		&  \gamma   \\
								\text{s.t. }	&\max\limits_{i\le n}\| r(d_i,0) \|  \le \gamma \vspace{1mm} \\ 
												& f \mapsto r(d,f) \neq 0,  \quad \forall d \in \D,
						\end{array} \right.
		\end{align}}
	Notice that the optimization problems $\APr$ and $\CPr$ are naturally stochastic as they depend on the generated scenarios $(d_i)_{i=1}^n$, which is indeed a random variable defined on $n$-fold product probability space $(\Omega^n, \sigalg^n, \PP^n)$. Therefore, their solutions are also random variables. In this work, we first restrict the FDI filters to a class of linear operators in which the random programs \eqref{prob-rand} are effectively convex, and hence tractable. In this step, the FDI filter is essentially robustified to $n$ signatures of the dynamic nonlinearity. Subsequently, invoking existing results on randomized optimization, in particular \cite{ref:Hansen-2012,ref:MohSut-13}, we will provide probabilistic guarantees on the relation of programs \eqref{rcp-ccp} and their probabilistic counterparts in \eqref{prob-rand}, whose precision is characterized in terms of the number of scenarios $n$. 
	
	We should highlight that the true empirical approximation of the chance constraint in $\CP$ is indeed ${1 \over n}\sum_{i=1}^{n} \ind{\big\{||r(d_i,0)|| \leq \gamma \big\}} \geq 1- \eps$, where $\ind{}$ is the indicator function. This approximation, as opposed to the one proposed in \eqref{prob-rand}, leads to a non-convex optimization program which is, in general, computationally intractable. In addition, note that the design parameter $\eps$ of $\Cp$ in \eqref{rcp-ccp} does not explicitly appear in the random counterpart $\CPr$ in \eqref{prob-rand}. However, as we will clarify in \ref{FDI:subsec:perf}, the parameter $\eps$ contributes to the probabilistic guarantees of the design. 

%==========================================================================================
\section{Model Description and Basic Definitions} \label{FDI:sec:model}
%==========================================================================================
	In this section we introduce a class of nonlinear models along with some basic definitions, which will be considered as the system dynamics in Figure \ref{fig:sys} throughout the article. Consider the nonlinear differential-algebraic equation (DAE) model 
	\begin{align}\label{model}
		E(x) + H(p)x + L(p)z + F(p) f = 0,
	\end{align}
	where the signals $x, z, f$ are assumed to be piece-wise continuous (p.w.c.) functions from $\R_{+}$ into $\R^{n_x}, \R^{n_z}, \R^{n_f}$, respectively; we denote the spaces of such signals by $\W^{n_x}, \W^{n_z}, \W^{n_f}$, respectively. Let $n_r$ be the number of rows in \eqref{model}, and $E:\R^{n_x} \ra \R^{n_r}$ be a Lipschitz continuous mapping. The operator $p$ is the distributional derivative operator \cite[Section I]{ref:Adams_Sobolev}, and $H, L, F$ are polynomial matrices in the operator $p$ with $n_r$ rows and $n_x, n_z, n_f$ columns, respectively. In the setup of Figure \ref{fig:sys}, the signal $x$ represents all unknowns signals, e.g., internal states of the system dynamics and unknown disturbances $d$. The signal $z$ contains all known signals, i.e., it is an augmented signal including control input $u$ and available measurements $y$. The signal $f$ stands for faults or intrusion which is the target of detection. We refer to \cite{ref:Shcheglova-DAE} and the references therein for general theory of nonlinear DAE systems and the regularity of their solutions. 
	
	One may extend the space of functions $x, z, f$ to Sobolev spaces, but this is outside the scope of our study. On the other hand, if these spaces are restricted to the (resp.\ right) smooth functions, then the operator $p$ can be understood as the classical (resp.\ right) differentiation operator. Throughout this article we will focus on continuous-time models, but one can obtain similar results for discrete-time models by changing the operator $p$ to the time-shift operator. We will think of the matrices $H(p)$, $L(p)$ and $F(p)$ above either as linear operators on the function spaces (in which case $p$ will be interpreted as a generalized derivative operator as explained above) or as algebraic objects (in which case $p$ will be interpreted as simply a complex variable). The reader is asked to excuse this slight abuse of the notation, but the interpretation should be clear from the context.
	
	Let us first show the generality of the DAE framework of \eqref{model} by the following example. Consider the classical nonlinear ordinary differential  equation 
	\begin{align}\label{ode}
	    \begin{cases}
	    G\dot{X}(t) &= E_X\big(X(t),d(t)\big) + AX(t) + B_u u(t) + B_d d(t) + B_f f(t)\\
	    Y(t) &= E_Y\big(X(t),d(t)\big) + CX(t) + D_u u(t) + D_d d(t) + D_f f(t)
	    \end{cases}
	\end{align}
	where $u(\cdot)$ is the input signal, $d(\cdot)$ the unknown disturbance, $Y(\cdot)$ the measured output, $X(\cdot)$ the internal variables, and $f(\cdot)$ a faults (or an attack) signal to be detected. Parameters $G$, $A$, $B_u$, $B_d$, $B_f$, $D_u$, $D_d$, and $D_f$ are constant matrices and functions $E_X, E_Y$ are Lipschitz continuous mappings with appropriate dimensions. One can easily fit the model  \eqref{ode} into the DAE framework of \eqref{model} by defining
	\begin{align*}
	     x &\Let \begin{bmatrix}
	                  X \\
	                  d \\
	               \end{bmatrix},& %\quad
	     z & \Let \begin{bmatrix}
	                  Y \\
	                  u \\
	               \end{bmatrix},& \\
	    E(x) \Let \begin{bmatrix}
				    E_X(x)\\
				    E_Y(x)
				 \end{bmatrix},   \quad  
	     H(p) & \Let \begin{bmatrix}
	                  -pG + A & B_d \\
	                  C & D_d \\
	               \end{bmatrix},& % \quad
	     L(p) &\Let \begin{bmatrix}
	                  0 & B_u \\
	                  -I & D_u \\
	               \end{bmatrix},& %\quad 
	    F(p) \Let \begin{bmatrix}
	                  B_f \\
	                  D_f \\
	               \end{bmatrix}.
	\end{align*}
	Following \cite{ref:NybergFrisk_TAC06}, with a slight extension to a nonlinear dynamics, let us formally characterize all possible observations of the model \eqref{model} in the absence of the fault signal $f$:
		\begin{equation}\label{M}
		    \M \Let \big\{ z \in \W^{n_z} \big | ~ \exists x \in \W^{n_x} : \quad E(x) + H(p)x + L(p)z = 0\big \};
		\end{equation}
		This set is known as the \emph{behavior} of the system \cite{ref:Polderman_BehavioralApproach98}. 
		
		\begin{Def}[Residual Generator]\label{def:residual}
		    A proper linear time invariant filter $r \Let R(p)z $ is a residual generator for \eqref{model} if for all $z \in \M$, it holds that $\lim\limits_{t \ra \infty} r(t) = 0$.
		\end{Def}
		Note that by Definition \ref{def:residual} the class of residual generators in this study is restricted to a class of \emph{linear} transfer functions where $R(p)$ is a matrix of proper rational functions of $p$. 
		\begin{Def}[Fault Sensitivity]\label{def:sensitivity}
		    The residual generator introduced in Definition \ref{def:residual} is sensitive to fault $f_i$ if the transfer function from $f_i$ to $r$ is nonzero, where $f_i$ is the $i^{th}$ elements of the signal $f$.
		\end{Def}
		One can inspect that Definition \ref{def:residual} and Definition \ref{def:sensitivity} essentially encode the basic mapping requirements \eqref{d-r} and \eqref{f-r}, respectively.

%==========================================================================================
\section{Fault Detection and Isolation Filters} \label{FDI:sec:FDI}
%==========================================================================================
	The main objective of this section is to establish a scalable framework geared towards the design perspectives $\Ap$ and $\Cp$ as explained in Section \ref{FDI:sec:prob}. To this end, we first review a polynomial characterization of the residual generators and its linear program formulation counterpart for linear systems (i.e., the case where $E(x) \equiv 0$). We then extend the approach to the nonlinear model \eqref{model} to account for the contribution of $E(\cdot)$ to the residual, and subsequently provide probabilistic performance guarantees for the resulting filter. 
	
	%------------------------------------------------------------------------------------------
	\subsection{Residual Generators for Linear Systems} \label{FDI:subsec:linear}
	%------------------------------------------------------------------------------------------
	In this subsection we assume $E(x) \equiv 0$, i.e., we restrict our attention to the class of linear DAEs.  One can observe that the behavior set $\M$ can alternatively be defined as
		\begin{equation*}
		    \M = \big\{ z \in \W^{n_z} \big |~ N_{H}(p)L(p)z = 0\big \},
		\end{equation*} 
	where the collection of the rows of $N_H(p)$ forms an irreducible polynomial basis for the left null-space of the matrix $H(p)$ \cite[Section 2.5.2]{ref:Polderman_BehavioralApproach98}. This representation allows one to describe the residual generators in terms of polynomial matrix equations. That is, by picking a linear combination of the rows of $N_H(p)$ and considering an arbitrary polynomial $a(p)$ of sufficiently high order with roots with negative real parts, we arrive at a residual generator in the sense of Definition \ref{def:residual} with transfer operator
		\begin{equation}\label{filter}
		    R(p) = a^{-1}(p) \gamma(p) N_H(p) L(p)  \Let  a^{-1}(p) N(p) L(p),
		\end{equation}
	where $\gamma(p)$ is a polynomial row vector representing a linear combination of the rows of $N_H(p)$. Note that the role of $\gamma(p)$ is implicitly taken into consideration by $N(p) \Let \gamma(p) N_H(p)$. The above filter can easily be realized by an explicit state-space description with input $z$ and output $r$. Multiplying the left hand-side of \eqref{model} by $a^{-1}(p)N(p)$ leads to 
			\begin{equation*}
				r = - a^{-1}(p)N(p)F(p)f.
			\end{equation*}
	Thus, a sensitive residual generator, in the sense of Definition \ref{def:residual} and Definition \ref{def:sensitivity}, is characterized by the polynomial matrix equations
		\begin{subequations}
		\label{poly}
			\begin{align}
			\label{poly1}  N(p) H(p) &= 0,\\
			\label{poly2}  N(p) F(p) &\neq 0,
			\end{align}
		\end{subequations}
	where \eqref{poly1} implements condition \eqref{d-r} above (cf. Definition \ref{def:residual}) while \eqref{poly2} implements condition \eqref{f-r} (cf. Definition \ref{def:sensitivity}). Both row polynomial vector $N(p)$ and denominator polynomial $a(p)$ can be viewed as design parameters. Throughout this study we, however, fix $a(p)$ and aim to find an optimal $N(p)$ with respect to a certain objective criterion related to the filter performance. 
	
	In case there are more than one faults ($n_f > 1$), it might be of interest to isolate the impact of one fault in the residual from the others. The following remark implies that the isolation problem is effectively a detection problem.
		
	\begin{Rem}[Fault Isolation]
	\label{rem:isol}
		Consider model \eqref{model} and suppose $n_f >1$. In order to detect only one of the fault signals, say $f_1$, and isolate it from the other faults, $f_i, i \in \{2,\cdots, n_f\}$, one may consider the detection problem for the same model but in new representation 
			\begin{align*}
				E(x) + [H(p) ~ \widetilde{F}(p)] \begin{bmatrix} x \\ \tilde{f} \end{bmatrix} + L(p)z + F_1(p) f = 0,
			\end{align*}
		where $F_1(p)$ is the first column of $F(p)$, and  $\widetilde{F}(p)\Let[F_2(p),\cdots,F_{n_f}(p)]$, and $\tilde{f}\Let[f_2,\cdots,f_{n_f}]$.
	\end{Rem}
	
	In light of Remark \ref{rem:isol}, one can build a bank of filters where each filter aims to detect a particular fault while isolating the impact of the others; see \cite[Theorem 2]{ref:Frisk_AUT09} for more details on fault isolation. Next, we show how to transform the matrix polynomial equations \eqref{poly} into a linear programming framework.
	
	\begin{Lem} %[Linear Programming Characterization]
	\label{lem:feas}
		Let $N(p)$ be a feasible polynomial matrix of degree $d_N$ for the inequalities \eqref{poly}, where
	\begin{align*}
		H(p) \Let \sum^{d_H}_{i=0} H_i p^i, \quad F(p) \Let \sum^{d_F}_{i=0} F_i p^i, \quad
		N(p) \Let \sum^{d_N}_{i=0} N_i p^i,
	\end{align*}
	and $H_i \in \R^{n_r \times n_x}$, $F_i \in \R^{n_r \times n_f}$, and $N_i \in \R^{1 \times n_r}$ are constant matrices. Then, the polynomial matrix inequalities \eqref{poly} are equivalent, up to a scalar, to
	\begin{subequations}
	\label{LP}
	    \begin{align}
	        \label{LP1} &\bar{N}\bar{H} = 0, \\
	        \label{LP2} &\big \| \bar{N}\bar{F} \big\|_{\infty}	 \ge 1,
	    \end{align}
	\end{subequations}
	where $\|\cdot\|_\infty$ is the infinity vector norm, and
	\begin{align*}
	    \bar{N} \Let
	    \begin{bmatrix}
	            N_0 & N_1 & \cdots & N_{d_N} \\
	    \end{bmatrix}
	\end{align*}
	\begin{align*}
	    \bar H \Let
	    \begin{bmatrix}
	          H_0 & H_1 & \cdots & H_{d_H} & 0 & \cdots & 0 \\
	          0 & H_0 & H_1 & \cdots & H_{d_H} & 0 & \vdots \\
	          \vdots &  & \ddots & \ddots &  & \ddots & 0 \\
	          0 & \cdots & 0 & H_0 & H_1 & \cdots & H_{d_H} \\
	    \end{bmatrix},
	\end{align*}
	\begin{align*}
	    \bar F \Let
	    \begin{bmatrix}
	          F_0 & F_1 & \cdots & F_{d_F} & 0 & \cdots & 0 \\
	          0 & F_0 & F_1 & \cdots & F_{d_F} & 0 & \vdots \\
	          \vdots &  & \ddots & \ddots &  & \ddots & 0 \\
	          0 & \cdots & 0 & F_0 & F_1 & \cdots & F_{d_F} \\
	    \end{bmatrix}.
	\end{align*}
	\end{Lem}

	\begin{proof}
	    It is easy to observe that
	    \begin{align*}
	        N(p)H(p) &= \bar{N}\bar{H} [I~pI~\cdots~p^{i}I]\tr, \qquad i \Let d_N + d_H, \\
	        N(p)F(p) &= \bar{N}\bar{F} [I~pI~\cdots~p^{j}I]\tr, \qquad j \Let d_N + d_F.
	    \end{align*}
	    Moreover, in light of the linear structure of equations \eqref{poly}, one can simply scale the inequality \eqref{poly2} and arrive at the assertion of the lemma.
	\end{proof}
	
%	\begin{Rem}\label{rem:inf norm}
%		Strictly speaking, the formulation in Lemma \ref{lem:feas} is not a linear program, due to the non-convex constraint \eqref{LP2}. It is, however, easy to see that if $\bar N$ is a solution to \eqref{LP}, then so is $-\bar N$. Hence, the inequality \eqref{LP2} can be understood as a family of $m$ linear programs where $m = n_f(d_F + d_N +1)$ is the number of columns of $\bar{F}$, and $n_f$ is the dimension of signal $f$ in \eqref{model}. Each of these linear programs focuses on a component of the vector $\bar N \bar F$, replacing the inequality \eqref{LP2} with
%	    \begin{align*}
%	        & \bar{N}\bar{F} v \ge 1, \qquad v \Let [0,\cdots, 1, \cdots, 0]\tr.
%	    \end{align*}
%	\end{Rem}
	
	Strictly speaking, the formulation in Lemma \ref{lem:feas} is not a linear program, due to the non-convex constraint \eqref{LP2}. It is, however, easy to show that the characterization \eqref{LP} can be understood as a number of linear programs, which grows linearly in the degree of the filter:
	
	\begin{Lem}\label{lem:inf norm}
		Consider the sets
		\begin{align*}
			\NN_j \Let \Big\{ \bar N \in \R^{n_r(d_N+1)}  ~ \big | ~ \bar N  \bar H = 0, ~ \bar N \bar F v_j \ge 1\Big \}, 
			\qquad v_j \Let \overset{ \quad \da ~j^{\text{th}} }{\big[0,\cdots, {1} , \cdots, 0 \big]\tr}, 
		\end{align*}
		and let $\NN \Let \bigcup_{j = 1}^m \NN_j$ where $m \Let n_f(d_F + d_N +1)$ is the number of columns of $\bar{F}$ (the parameters $\bar H, \bar F, n_f, d_F, d_N$ are as considered in Lemma \ref{lem:feas}). Then, the set characterized by \eqref{LP} is equivalent to $\NN \cup -\NN$. 
	\end{Lem}

	\begin{proof}
		Notice that $\| \bar N \bar F \|_\infty \ge 1$ if and only if there exists a coordinate $j$ such that $\bar N \bar F v_j \ge 1$ or $\bar N \bar F v_j \le -1$. Thus, the proof readily follows from the fact that each of the set $\NN_j$ focuses on a component of the vector $\bar N \bar F$ in \eqref{LP2}. 
	\end{proof}

	\begin{Fact}\label{fact:rank}
	    There exists a solution $N(p)$ to \eqref{poly} if and only if $\text{Rank}~[H(p) ~ F(p)] > \text{Rank}~H(p)$.
	\end{Fact}
	
	Fact \ref{fact:rank} provides necessary and sufficient conditions for the feasibility of the linear program formulation in Lemma \ref{lem:feas}; proof is omitted as it is an easy adaptation of the one in \cite[Corollary 3]{ref:Frisk_AUT09}.
	
	%------------------------------------------------------------------------------------------
	\subsection{Extension to Nonlinear Systems} \label{FDI:subsec:nonlinear}
	%------------------------------------------------------------------------------------------
	In the presence of nonlinear terms $E(x) \neq 0$, it is straightforward to observe that the residual of filter \eqref{filter} consists of two terms:
		\begin{align}\label{non-res}
			r \Let  R(p)z = - \underbrace{a^{-1}(p)N(p)F(p)f}_{(i)} - \underbrace{a^{-1}(p)N(p)E(x)}_{(ii)}.
		\end{align}
	Term (i) is the desired contribution of the fault $f$ and is in common with the linear setup. Term (ii) is due to the nonlinear term $E(\cdot)$ in \eqref{model}. Our aim here will be to reduce the impact of $E(x)$ while increasing the sensitivity to the fault $f$. To achieve this objective, we develop two approaches to control each of the two terms separately; in both cases we assume that the degree of the filter (i.e., $d_N$ in Lemma \ref{lem:feas}) and the denominator (i.e., $a(p)$ in \eqref{non-res}) are fixed, and the aim is to design the numerator coefficients (i.e., $N(p)$ in \eqref{non-res}).
		
	\subsubsection*{Approach (I) (Fault Sensitivity)}
	
	To focus on fault sensitivity while neglecting the contribution of the nonlinear term, we assume that the system operates close to an equilibrium point $x_e \in \R^{n_x}$. Even though in case of a fault the system may eventually deviate substantially from its nominal operating point, if the FDI filter succeeds in identifying the fault early the system will not have time to deviate too far. Hence, one may hope that a filter based on linearizing the system dynamics around the equilibrium would suffice. Then we assume, without loss of generality, that
	\begin{align*}
		\lim_{x \ra x_e} \frac{\big\|E(x)\big\|_2}{\|x-x_e\|_2} = 0, %\qquad \forall p \ge 1
	\end{align*}
	where $\|\cdot\|_2$ stands for the Euclidean norm of a vector. If this is not the case, the linear part of $E(\cdot)$ can be extracted and included in the linear part of the system.
	
	To increase the sensitivity of the linear filter to the fault $f$, we revisit the linear programming formulation \eqref{LP} and seek a feasible numerator $N(p)$ such that the coefficients of the transfer function $N(p)F(p)$ attain maximum values within the admissible range. This gives rise to the following optimization problem: 
		\begin{align} 
		\label{LP-opt}
		\left\{ \begin{array}{lll}
				\max\limits_{\bar N}		&  \big \|\bar{N}\bar{F} \big\|_{\infty}   \\
				\text{s.t. }	& \bar{N}\bar{H} &= 0 \vspace{1mm} \\
								& \big \| \bar{N} \big\|_{\infty} &\le 1
						\end{array} \right.
		\end{align}
	where the objective function targets the contribution of the signal $f$ to the residual $r$. Let us recall that $\bar{N}\bar{F}$ is the vector containing all numerator coefficients of the transfer function $f \mapsto r$. The second constraint in \eqref{LP-opt} is added to ensure that the solutions remain bounded; note that thanks to the linearity of the filter this constraint does not influence the performance. Though strictly speaking \eqref{LP-opt} is not a linear program, in a similar fashion as in Lemma \ref{lem:inf norm} it is easy to transform it to a family of $m$ different linear programs, where $m$ is the number of columns of $\bar F $. 
	
	%The main problem with this approach is the lack of prior information when the linearization technique is ``precise enough". That is, 
	How well the filter designed by \eqref{LP-opt} will work depends on the magnitude of the second term in \eqref{non-res}, which is due to the nonlinearities $E(x)$ and is ignored in \eqref{LP-opt}. If the term generated by $E(x)$ is large enough, the filter may lead to false alarms, whereas if we set our thresholds high to tolerate the disturbance generated by $E(x)$ in nominal conditions, the filter may lead to missed detections. A direct way toward controlling this trade-off involving the nonlinear term will be the focus of the second approach.  

	\subsubsection*{Approach (II) (Robustify to Nonlinearity Signatures)}
		
	This approach is the main step toward the theoretical contribution of the article, and provides the principle ingredients to tackle the proposed perspectives $\Ap$ and $\CP$ introduced in \eqref{rcp-ccp}. The focus is on term (ii) of the residual \eqref{non-res}, in relation to the mapping \eqref{d-r}. The idea is to robustify the filter against certain signatures of the nonlinearity during nominal operation. In the following we restrict the class of filters to the feasible solutions of polynomial matrix equations \eqref{poly}, characterized in Lemma \ref{lem:feas}.  
	
	Let us denote the space of all p.w.c. functions from the interval $[0,T]$ to $\R^{n}$ by $\W^n_T$. We equip this space with the $\Lnorm$-inner product and the corresponding norm %$ \| e \|^2_{\Lnorm} \Let \inner{e,e} = \int_0^T e\tr(t)e(t) \diff t$.	
		\begin{align*}
		%\label{<>}
		 \| e \|_{\Lnorm} \Let \sqrt{\inner{e,e}}, \qquad \qquad \inner{e,g} \Let  \int_0^T e\tr(t)g(t) \diff t, \quad e, g \in \W^n_T. 
		\end{align*}
	Consider an unknown signal $x \in \W^{n_x}_T$. In the context of the ODEs \eqref{ode} that means we excite the system with the disturbance $d(\cdot)$ for the time horizon $T$. We then stack $d(\cdot)$ together with the internal state $X(\cdot)$ to introduce $x \Let [\begin{smallmatrix} X \\ d \end{smallmatrix}]$. We define the signals $e_x \in \W^{n_r}_T$ and $ r_x \in \W^1_T$ as follows:
		\begin{align}
		\label{r_x}
			e_x(t) \Let E\big(x(t)\big), \qquad 
			r_x(t) \Let -a^{-1}(p)N(p)[e_x](t) , \qquad \forall t \in [0,T].
		\end{align}
	The signal $e_x$ is the ``\emph{nonlinearity signature}" in the presence of the unknown signal $x$, and the signal $r_x$ is the contribution of the nonlinear term to the residual of the linear filter. Our goal now is to minimize  $\|r_x\|_{\Lnorm}$ in an optimization framework in which the coefficients of polynomial $N(p)$ are the decision variables and the denominator $a(p)$ is a fixed stable polynomial with the degree at least the same as $N(p)$. 
	
	\begin{Lem}
	\label{lem:psi}
		Let $N(p)$ be a polynomial row vector of dimension $n_r$ and degree $d_N$, and $a(p)$ be a stable scalar polynomial with the degree at least $d_N$. For any $x \in \W^{n_x}_T$ there exists $\psi_x \in \W_T^{n_r(d_N+1)}$ such that
		%Consider an unknown signal $x \in \W^{n_x}_T$ and the corresponding residual $r_x$ as defined in \eqref{r_x}. Then, there exists a mapping $\psi_x \in \W_T^{n_r(d_N+1)}$ independent of $N(p)$ such that $r_x = \bar N \psi_x$,
%			\begin{enumerate}[label=(\roman*), leftmargin = *, itemsep=2mm, parsep=0mm]
%				\item $r_x(t) = \bar N \psi_x(t), \qquad \forall t \in [0,T]$,
%				\item $\| \psi_x \|_{\Lnorm} \le C \|x\|_{\Lnorm}$,
%			\end{enumerate}			
		\begin{subequations}
		\label{psi}
			\begin{align}
				\label{psi-rx} r_x(t) &= \bar N \psi_x(t),& & \forall t \in [0,T]& \\
				\label{psi-norm} \| \psi_x \|_{\Lnorm} &\le C \|e_x\|_{\Lnorm},& & C \Let \sqrt{n_r(d_N + 1)}\|a^{-1}\|_{\Hinf},&
				%\max \big\{ \|a^{-1}\|_{\Hinf}, 1 \big\},
			\end{align}					
		\end{subequations}	 
%		\begin{align*}
%			\forall x \in \W^{n_x}_T, ~ \exists \psi_x \in \W^{n_r(d_N+1)}_T ~ : \qquad r_x(t) = \bar N \psi_x(t), \quad \forall t \in [0,T],
%		\end{align*}
		where $\bar N$ is the vector collecting all the coefficients of the numerator $N(p)$ as introduced in Lemma \ref{lem:feas}, and the signals $e_x$ and $r_x$ are defined as in \eqref{r_x}. %\footnote{$\|G\|_{\Hinf} \Let \sup_{\omega \in \R} \overline{\sigma}\big(G(j\omega)\big)$ is the $\Hinf$-norm of a matrix transfer function, where $\overline{\sigma}$ denotes the maximum singular value of the matrix $G(j\omega)$.}
	\end{Lem}
	\begin{proof}
		See Appendix \ref{FDI:subsec:pf:nonlinear}.
	\end{proof}
	
	Given $x \in \W^{n_x}_T$ and the corresponding function $\psi_x$ as defined in Lemma \ref{lem:psi}, we have 
	\begin{align}
		\label{Q_x}
			\|r_x\|^2_{\Lnorm} = \bar N Q_x \bar N\tr, \qquad Q_x \Let \int_0^T \psi_x(t) \psi_x\tr(t) \diff t. 
	\end{align}
	We call $Q_x$ the ``\emph{signature matrix}" of the nonlinearity signature $t \mapsto e_x(t)$ resulting from the unknown signal $x$. Given $x$ and the corresponding signature matrix $Q_x$, the $\Lnorm$-norm of $r_x$ in \eqref{r_x} can be minimized by considering an objective which is a quadratic function of the filter coefficients $\bar N$ subject to the linear constraints in \eqref{LP}: 
		\begin{align} 
		\label{QP}
		\left\{ \begin{array}{lll}
				\min\limits_{\bar N}		&  \bar{N} Q_x \bar{N}\tr   \\
				\text{s.t. }	& \bar{N}\bar{H} &= 0 \vspace{1mm} \\
								& \big \| \bar{N}\bar{F} \big\|_{\infty} &\ge 1
						\end{array} \right.
		\end{align}
	The program \eqref{QP} is not a true quadratic program due to the second constraint. Following Lemma \ref{lem:inf norm}, however, one can show that the optimization program \eqref{QP} can be viewed as a family of $m$ quadratic programs where $m = n_f(d_F + d_N + 1)$. %Let us highlight once again that the signature matrix $Q_x$ in \eqref{QP} concerns the nonlinearity signature of the dynamics in the presence of a certain disturbance pattern, whereas the linear parts of the dynamics are incorporated into the constraints, in particular $\bar H$. 
	
	In the rest of the subsection, we establish an algorithmic approach to approximate the matrix $Q_x$ for a given $x \in \W^{n_x}_T$, with an arbitrary high precision. We first introduce a finite dimensional subspace of $\W^{1}_T$ denoted by
		\begin{align}
		\label{B}
			\B\Let span\{b_0,b_1, \cdots, b_k\},
		\end{align}	
	where the collection of $b_i: [0,T]\ra \R$ is a basis for $\B$. Let $\B^{n_r} \Let \bigotimes_{i=1}^{n_r} \B$ be the $n_r$ Cartesian product of the set $\B$, and $\proj : \W^{n_r}_T \ra \B^{n_r}$ be the $\Lnorm$-orthogonal projection operator onto $\B^{n_r}$, i.e., 
	\begin{align}
	\label{T-proj}
		\proj(e_x) = \sum_{i=0}^k \beta_{i}^\star b_i, \qquad \beta^\star \Let \arg\min_{\beta} \big\| e_x - \sum\limits_{i=0}^{k} \beta_i b_i\big\|_{\Lnorm}
	\end{align}
	Let us remark that if the basis of $\B$ is orthonormal (i.e., $\inner{b_i,b_j} = 0$ for $i \neq j$), then $\beta_i^\star = \int\limits_{0}^{T} b_i(t) e_x(t) \diff t$; we refer to \cite[Section 3.6]{ref:Luenberger_OptVecS} for more details on the projection operator.
	
	\begin{As}\label{a:basis}
	We stipulate that 
		\begin{enumerate} [label=(\roman*), leftmargin = *, itemsep=2mm, parsep=0mm]
			\item \label{a:basis-smooth} The basis functions $b_i$ of subspace $\B$ are smooth and $\B$ is closed under the differentiation operator $p$, i.e., for any $b \in \B$ we have $p[b] = \frac{\diff }{\diff t}b \in \B$. 			
		
			\item \label{a:basis-comp} The basis vectors in \eqref{B} are selected from an $\Lnorm$-complete basis for $\W^{1}_T$, i.e., for any $e \in \W^{n_r}_T$, the projection error $\big\| e - \proj(e) \big\|_{\Lnorm}$ can be made arbitrarily small by increasing the dimension $k$ of subspace $\B$.
		\end{enumerate}
	\end{As}
	
	The requirements of Assumptions \ref{a:basis} can be fulfilled for subspaces generated by, for example, the polynomial or Fourier basis. Thanks to Assumption \ref{a:basis}\ref{a:basis-smooth}, the linear operator $p$ can be viewed as a matrix operator. That is, there exists a square matrix $D$ with dimension $k+1$ such that
		\begin{align}
		\label{D}
			p [B(t)] = \frac{\diff}{\diff t}B(t) = D B(t), \qquad B(t)\Let[b_0(t), \cdots, b_k(t)]\tr.
		\end{align}
	In Section \ref{FDI:subsec:sim:filter} we will provide an example of such matrix operator for the Fourier basis. 
%	Let $\beta \Let [\beta_{0},\cdots,\beta_{k}]$ be the representation of $\proj(e_x)$ in terms of the basis $(b_i)_{i=1}^k$, i.e., 
%		\begin{align}
%		\label{beta}
%			\proj(e_x) = \beta B = \sum_{i=0}^k \beta_{i} b_i,
%		\end{align}
%	where $\beta_i$ are constant vectors in $\R^{n_r}$. 
	By virtue of the matrix representations of \eqref{D} we have
		\begin{align}
		\label{D-bar}
			N(p) \proj(e_x) = \sum_{i=0}^{d_N} N_i p^i \beta^\star B = \sum_{i=0}^{d_N} N_i \beta^\star D^i B = \bar N \bar D B, \qquad \bar D \Let \begin{bmatrix} \beta^\star\\ \beta^\star D \\ \vdots \\ \beta^\star D^{d_N} \end{bmatrix},
		\end{align}
	where the vector $\beta^\star \Let [\beta^\star_{0},\cdots,\beta^\star_{k}]$ is introduced in \eqref{T-proj}. If we define the positive semidefinite matrix $G \Let [G_{ij}]$ of dimension $k+1$ by
	\begin{align}
		\label{G}
		G_{ij} \Let \inner{a^{-1}(p)[b_i] , a^{-1}(p)[b_j]}, %\int_0^{T} [a^{-1}(p)B](t) [a^{-1}(p)B](t) \diff t,
	\end{align}
	%where $G_{ij}$ denotes the elements of row $i$ and column $j$ of $G$, and $b_i$ are the basis functions in \eqref{B}. %Therefore, with the aid of equations \eqref{D-bar} and \eqref{G} 
	we arrive at
		\begin{align}
		\label{Q}
			 \big \| a^{-1}(p)N(p)\proj(e) \big \|^2_{\Lnorm} = \bar N Q_{\B} {\bar N}\tr, \qquad Q_{\B} \Let \bar D G {\bar D}\tr,
			 %\Let \inner{N(p)\proj(e),N(p)\proj(e)}
		\end{align}
	where $\bar D$ and $G$ are defined in \eqref{D-bar} and \eqref{G}, respectively. Note that the matrices $G$ and $D$ are built by the data of the subspace $\B$ and denominator $a(p)$, whereas the nonlinearity signature only influences the coefficient $\beta^\star$. The above discussion is summarized in Algorithm \ref{alg} with an emphasis on models described by the ODE \eqref{ode}, while Proposition \ref{prop:app} addresses the precision of the approximation scheme. %The precision of the algorithm output is quantified in Theorem \ref{thm:app}. 

	\begin{algorithm}[t!]
	\caption{Computing the signature matrix $Q_x$ in \eqref{Q_x}}
	\label{alg}
	\begin{algorithmic} 
	\State
	\begin{enumerate} [label=$(\roman*)$, itemsep = 2mm, topsep = -2mm, leftmargin = 5mm]
	\item {\bf Initialization of the Filter Paramters:}
		\begin{enumerate} %[label=$\arabic*.$, itemsep = -1mm, topsep = -2mm, leftmargin = 5mm]
			\item Select a stable filter denominator $a(p)$, a numerator degree $d_N$ not higher than $a(p)$ order, and horizon $T$
			\item Select a basis $\{b_i\}_{i=1}^{k} \subset \W_T^1$ satisfying Assumptions \ref{a:basis}
			\item Compute the differentiation matrix $D$ in \eqref{D}
			\item Compute the matrix $G$ in \eqref{G} \footnotemark
		\end{enumerate}
	\item {\bf Identification of the Nonlinearity Signature:}
		\begin{enumerate} %[label=$\arabic*.$, itemsep = -1mm, topsep = -2mm, leftmargin = 5mm]
			\item \underline{Input} the disturbance pattern $d(\cdot)$ for time horizon $T$
			\item Solve \eqref{ode} under inputs $d(\cdot)$ and $f \equiv 0$ to obtain the internal state $X(\cdot)$
			\item Set the unknown signal $x(t) \Let [X\tr(t),~ d\tr(t)]\tr$
			\item Set the nonlinearity signature $e_x(t) \Let \big[ E_X\tr\big(x(t)\big),~ E_Y\tr\big(x(t)\big)\big]\tr$
		\end{enumerate}
	\item {\bf Computation of the Signature Matrix}
		\begin{enumerate} %[label=$\arabic*.$, itemsep = -1mm, topsep = -2mm, leftmargin = 5mm]
			\item Compute $\beta^\star$ from \eqref{T-proj} (in case of orthonormal basis $\beta_i^\star = \int\limits_{0}^{T} b_i(t) e_x(t) \diff t$)
			\item Compute $\bar{D}$ from \eqref{D-bar}
			\item \underline{Ouput} $Q_\B \Let \bar D G {\bar D}\tr$ in \eqref{Q}
		\end{enumerate}
	\end{enumerate}
	\end{algorithmic}
	\end{algorithm}
	
		\begin{Prop}[Signature Matrix Approximation]
		\label{prop:app}
			Consider an unknown signal $x :[0,T] \ra \R^{n_x}$ in $\W^{n_x}_T$ and the corresponding nonlinearity signature $e_x$ and signature matrix $Q_x$ as defined in \eqref{r_x} and \eqref{Q_x}, respectively. Let $(b_i)_{i \in \N} \subset \W^{1}_T$ be a family of basis functions satisfying Assumptions \ref{a:basis}, and let $\B$ be the finite dimensional subspace in \eqref{B}. If $\| e_x - \proj(e_x)\|_{\Lnorm} < \delta$, where $\proj$ is the projection operator onto $\B^{n_r}$, then
			\begin{align}
			\label{QP-eps}
				%\| e_x - \proj(e_x)\|_{\Lnorm} \le \eps \quad \Ra \quad 
				\big \| Q_x - Q_{\B} \big \|_2 < \bar{C} \delta, % \| e_x - \proj(e_x)\|_{\Lnorm} ,	
				\qquad \bar C \Let \big(1+2\|e_x\|_{\Lnorm}\big)C \|a^{-1}\|_{\Hinf} ,			
			\end{align}
			where $Q_{\B}$ is obtained by \eqref{Q} (the output of Algorithm \ref{alg}), and $C$ is the same constant as in \eqref{psi-norm}.
		\end{Prop}
		
		\begin{proof}
			See Appendix \ref{FDI:subsec:pf:nonlinear}.	
		\end{proof}
		
	%We close this subsection with a following remark, a natural extension to robustify the filter to multiple nonlinearity signatures.  
	
	\begin{Rem}[Multi Signatures Training]
		\label{rem:multi}
			In order to robustify the FDI filter to more than one unknown signal, say $\{x_i(\cdot)\}_{i=1}^{n}$, one may introduce an objective function as an average cost $\bar N \big(\frac{1}{n}\sum_{i=1}^n Q_{x_i}\big)\bar{N}\tr$ or the worst case viewpoint $\max_{i\le n} \bar{N} Q_{x_i} \bar{N}\tr$, where $Q_{x_i}$ is the signature matrix corresponding to $x_i$ as defined in \eqref{Q_x}.
	\end{Rem}

	%------------------------------------------------------------------------------------------
	\subsection{Proposed Methodology and Probabilistic Performance} \label{FDI:subsec:perf}
	%------------------------------------------------------------------------------------------
	
	The preceding subsection proposed two optimization-based approaches to enhance the FDI filter design from linear to nonlinear system dynamics. Approach (I) targets the fault sensitivity while neglecting the nonlinear term of the system dynamics, and Approach (II) offers a QP framework to robustify the residual with respect to signatures of the dynamic nonlinearities. Here our aim is to achieve a reconciliation between these two approaches. We subsequently provide theoretical results from the proposed solutions to the original design perspectives \eqref{rcp-ccp}.

		{\footnotetext{A conservative but easy-to-implement approach is to set $G$ an identity matrix with dimension $k+1$.}}

	Let $(d_i)_{i=1}^n \subset \D$ be i.i.d.\ disturbance patterns generated from the probability space $(\Omega, \sigalg, \PP)$. For each $d_i$, let $x_i$ be the corresponding unknown signal with the associated signature matrix $Q_{x_i}$ as defined in \eqref{Q_x}. In regard to the average perspective $\Ap$, we propose the two-stage (random) optimization program
	\begin{subequations}
		\label{APr}
			\begin{align} 
			\label{APr_1}
				\APr_1:& \left\{ \begin{array}{lll}
					\min\limits_{\gamma, \bar N}		&  \gamma   \\
					\text{s.t. }	& \bar{N}\bar{H}  &= 0 \vspace{1mm} \\
									& \big \| \bar{N}\bar{F} \big\|_{\infty} &\ge 1 \vspace{1mm}\\
									& \frac{1}{n}\sum\limits_{i=1}^n J \Big( \sqrt{\bar N Q_{x_i} \bar{N}\tr} \Big) &\le \gamma
							\end{array} \right.
				%\qquad 
%				&\CPr_1:& \left\{ \begin{array}{lll}
%					\min\limits_{\gamma, \bar N}		&  \gamma   \\
%					\text{s.t. }	& \bar{N}\bar{H} &= 0 \vspace{1mm} \\
%									& \big \| \bar{N}\bar{F} \big\|_{\infty} &\ge 1 \vspace{2mm}\\
%									& \max\limits_{i\le n} {\bar{N} Q_{x_i} \bar{N}\tr} & \le \gamma
%							\end{array} \right. 
			\\ %\vspace{2mm}
			\label{APr_2}
				\APr_2:& \left\{ \begin{array}{lll}
					\max\limits_{\bar N}		&   \big \| \bar{N}\bar{F} \big\|_{\infty}  \\
					\text{s.t. }	& \bar{N}\bar{H} &= 0 \vspace{1mm} \\
									& \big \| \bar{N} \big\|_{\infty} & \le 1 \vspace{1mm}\\
									& \frac{1}{n}\sum\limits_{i=1}^n J \Big({\| \bar N_1\opt\|_\infty}{\sqrt{\bar N Q_{x_i} \bar{N}\tr}} \Big) & \le {\gamma_1\opt}
							\end{array} \right.
				%\qquad 
%				&\CPr_2:& \left\{ \begin{array}{lll}
%					\max\limits_{\bar N}		&  \big \| \bar{N}\bar{F} \big\|_{\infty}    \\
%					\text{s.t. }	& \bar{N}\bar{H} &= 0 \vspace{1mm} \\
%									& \big \|\bar{N} \big\|_{\infty} &\le {\| \bar N_1\opt\|_\infty}  \vspace{2mm}\\
%									& %{\| \bar N_1\opt\|^2_\infty} 
%									\max\limits_{i\le n} \bar{N}Q_{x_i} \bar{N}\tr & \le {\gamma_1\opt}
%							\end{array} \right.
			\end{align}
	\end{subequations}
	where  $J:\R_+ \ra \R_+$ is an increasing and convex payoff function, and in the second stage \eqref{APr_2} $\bar N_1\opt$ and $\gamma_1\opt$ are the optimizers of the first stage \eqref{APr_1}, i.e., the programs \eqref{APr} need to be solved sequentially in a \emph{lexicographic} (multi-objective) sense \cite{ref:Lex-04}. Let us recall that the filter coefficients can always be normalized with no performance deterioration. Hence, it is straightforward to observe that the main goal of the second stage is only to improve the coefficients of $\bar{N}\bar{F}$ (concerning the fault sensitivity) while the optimality of the first stage (concerning the robustification to nonlinearity signatures) is guaranteed. Similarly, we also propose the following two-stage program for the perspective $\Cp$:
	\begin{subequations}
		\label{CPr}
			\begin{align} 
			\label{CPr_1}
				\CPr_1:& \left\{ \begin{array}{lll}
					\min\limits_{\gamma, \bar N}		&  \gamma   \\
					\text{s.t. }	& \bar{N}\bar{H} &= 0 \vspace{1mm} \\
									& \big \| \bar{N}\bar{F} \big\|_{\infty} &\ge 1 \vspace{2mm}\\
									& \max\limits_{i\le n} {\bar{N} Q_{x_i} \bar{N}\tr} & \le \gamma
							\end{array} \right. 
			\\ %\vspace{2mm}
			\label{CPr_2}
				\CPr_2:& \left\{ \begin{array}{lll}
					\max\limits_{\bar N}		&  \big \| \bar{N}\bar{F} \big\|_{\infty}    \\
					\text{s.t. }	& \bar{N}\bar{H} &= 0 \vspace{1mm} \\
									& \big \|\bar{N} \big\|_{\infty} &\le 1 \vspace{2mm}\\
									& %{\| \bar N_1\opt\|^2_\infty} 
									{\| \bar N_1\opt\|^2_\infty} \big(\max\limits_{i\le n} \bar{N}Q_{x_i} \bar{N}\tr\big) & \le {\gamma_1\opt}
							\end{array} \right.
			\end{align}
		\end{subequations}

	\begin{Rem}[Computational Complexity]
	\label{rem:AP}
		In view of Lemma \ref{lem:inf norm}, all the programs in \eqref{APr} and \eqref{CPr} can be written as families of convex programs, and hence are tractable. It is, however, worth noting that in case the payoff function of $\APr$ is $J(\alpha) \Let \alpha^2$, the computational complexity of the resulting programs in \eqref{APr} is independent of the number of scenarios $n$, since the problems effectively reduce to a quadratic programming with a constraint involving the average of all the respective signature matrices (i.e., ${1 \over n}\sum_{i=1}^{n}Q_{x_i}$). This is particularly of interest if one requires to train the filter for a large number of scenarios. 
	\end{Rem}
	
	Clearly, the filter designed by programs \eqref{APr} and \eqref{CPr} is robustified to only finitely many most likely events, and as such, it may remain sensitive to disturbance patterns which have not been observed in the training phase. However, thanks to the probabilistic guarantees detailed in the sequel, we shall show that the probability of such failures (\emph{false alarm}) is low. In fact, the tractability of our proposed scheme comes at the price of allowing for rare threshold violation of the filter. The rest of the subsection formalizes this probabilistic bridge between the program \eqref{APr} (resp.\ \eqref{CPr}) and the original perspective $\Ap$ (resp.\ $\Cp$) in \eqref{rcp-ccp} when the class of filters is confined to the linear residuals characterized in Lemma \ref{lem:feas}. For this purpose, we need a technical measurability assumption which is always expected to hold in practice. 

	\begin{As}[Measurability]
		\label{a:x:meas}
		We assume that the mapping $\D \ni d \mapsto x \in \W^{n_x}_T$ is measurable where the function spaces are endowed with the $\Lnorm$-topology and the respective Borel sigma-algebra. In particular, $x$ can be viewed as a random variable on the same probability space as $d$. 
	\end{As}
	
	Assumption \ref{a:x:meas} is referred to the behavior of the system dynamics as a mapping from the disturbance $d$ to the internal states. In the context of ODEs \eqref{ode}, it is well-known that under mild assumptions (e.g., Lipschitz continuity of $E_X$) the mapping $d \mapsto X$ is indeed continuous \cite[Chapter 5]{ref:Khalil}, which readily ensures Assumption \ref{a:x:meas}.  
	
	\subsubsection{Probabilistic performance of $\APr$}
	
	Here we study the asymptotic behavior of the empirical average of $\EE{\big[J(\|r\|)\big]}$ uniformly in the filter coefficients $\bar N$, which allows us to link the solutions of programs \eqref{APr} to $\Ap$. Let $\NN \Let \big\{\bar N \in \R^{nr(d_N+1)} : \| \bar N \|_{\infty} \le 1 \big\}$ and consider the payoff function of $\Ap$ in \eqref{rcp-ccp} as the mapping $\phi : \NN \times \W^{n_x}_T \ra \R_+$:
		\begin{align}
			\label{phi}
			\phi(\bar N,x) \Let J\big(\| r_x \|_{\Lnorm}\big) = J \big( \| \bar N \psi_x \|_{\Lnorm}\big),
		\end{align}
	where the second equality follows from Lemma \ref{lem:psi}. 
	
	\begin{Thm}[Average Performance]
	\label{thm:AP}
		Suppose Assumption \ref{a:x:meas} holds and the random variable $x$ is almost surely bounded\footnote{This assumption may be relaxed in terms of the moments of $x$, though this will not be pursued further here.}. Then, the mapping $\bar N \mapsto \phi(\bar N, x)$ is a random function. Moreover, if $(x_i)_{i=1}^n \subset \W^{n_x}_T$ are i.i.d.\ random variables and $e_n$ is the uniform empirical average error 
		\begin{align}
		\label{e_n}
			%\wt{\phi}_n(\bar N) \Let \frac{1}{n}\sum\limits_{i=1}^{n} \phi(\bar N,x_i), \qquad 
			e_n \Let \sup_{\bar N \in \NN} \Big\{ \frac{1}{n}\sum\limits_{i=1}^{n} \phi(\bar N,x_i) - \EE \big[\phi(\bar N, x) \big] \Big\},
		\end{align}
		then, 
		\begin{enumerate} [label=(\roman*), leftmargin = *, itemsep=2mm, parsep=0mm]
			\item \label{slln}  the {Strong Law of Large Numbers} (SLLN) holds, i.e., $ \lim\limits_{n \ra \infty} e_n = 0$ almost surely. 
			
			\item \label{uclt} the {Uniform Central Limit Theorem} (UCLT) holds, i.e., $\sqrt{n}e_n$ converges in law to a Gaussian variable with distribution $N(0,\sigma)$ for some $\sigma \ge 0$. 
		\end{enumerate}
	\end{Thm}
	
	\begin{proof}
		%For the proof of \ref{slln} (resp.\ \ref{uclt}), it suffices to verify that the required conditions of \cite[Thm.\ 2.1]{ref:Hansen-2012} (resp.\ \cite[Thm.\ 6.3.3, p.\ 208]{ref:Dudley99}) are met; 
		See Appendix \ref{FDI:subsec:pf:perf} along with required preliminaries.
	\end{proof}
	
	The following Corollary is an immediate consequence of the UCLT in Theorem \ref{thm:AP} \ref{uclt}. 
	\begin{Cor}
		\label{cor:AP}
		Let assumptions of Theorem \ref{thm:AP} hold, and $e_n$ be the empirical average error \eqref{e_n}. For all $\eps > 0$ and $k < \frac{1}{2}$, we have 
		\begin{align*}
			\lim_{n \ra \infty} \PP^n \big( n^ke_n \ge \eps \big) = 0,
		\end{align*}
		where $\PP^n$ denotes the $n$-fold product probability measure on $\big(\Omega^n, \sigalg^n \big)$.
	\end{Cor}
	
	\subsubsection{Probabilistic performance of $\CPr$}
	
	The formulation $\Cp$ in \eqref{rcp-ccp} is known as chance constrained program which has received increasing attention due to recent developments toward tractable approaches, in particular via the scenario counterpart (cf. $\CPr$ in \eqref{prob-rand}) in a convex setting \cite{ref:CalCam-06,ref:CamGar-08}. These studies are, however, not directly applicable to our problem due to the non-convexity arising from the constraint $\|\bar N \bar F \|_\infty \ge 1$. Here, following our recent work \cite{ref:MohSut-13}, we exploit the specific structure of this non-convexity and adapt the scenario approach accordingly. 

	Let $\big(\bar N_n \opt, \gamma_n\opt\big)$ be the optimizer obtained through the two-stage programs \eqref{CPr} where $\bar N_n \opt$ is the filter coefficients and $\gamma_n\opt$ represents the filter threshold; $n$ is referred to the number of disturbance patterns. Given the filter $\bar N_n \opt$, let us denote the corresponding filter residual due to the signal $x$ by $r_x[\bar N_n^*]$; this is a slight modification of our notation $r_x$ in \eqref{r_x} to specify the filter coefficients. To quantify the filter performance, one may ask for the probability that a new unknown signal $x$ violates the threshold $\gamma_n\opt$ when the FDI filter is set to $\bar N_n \opt$  (i.e., the probability that {\small $\big \|r_x[\bar N_n^*] \big \|^2_{\Lnorm} > {\gamma_n^*}$}). In the FDI literature such a violation is known as a false alarm, and from the $\CP$ standpoint its occurrence probability is allowed at most to the $\eps$ level. In this view the performance of the filter can be quantified by the event
		\begin{align}
			\label{event}
				\mathcal{E}\big(\bar N_n \opt, \gamma^*_n\big) \Let \Big \{ \PP \Big( \big \|r_x[\bar N_n^*] \big \|^2_{\Lnorm} > \gamma_n^*  \Big) > \eps \Big \}. 
		\end{align}
	The event \eqref{event} accounts for the feasibility of the $\CPr$ solution from the original perspective $\CP$. Note that the measure $\PP$ in \eqref{event} is referred to $x$ whereas the stochasticity of the event stems from the random solutions $\big(\bar N_n \opt, \gamma^*_n\big)$.\footnote{The measure $\PP$ is, with slight abuse of notation, the induced measure via the mapping addressed in Assumption \ref{a:x:meas}.} %We proceed with the main result of this part in regard to the likelihood of the event $\eqref{event}$. 

		\begin{Thm}[Chance Performance]
		\label{thm:CP}
			Suppose Assumption \ref{a:x:meas} holds and $(x_i)_{i=1}^n$ are i.i.d.\ random variables on $(\Omega,\sigalg,\PP)$. Let $\bar N_n \opt \in \R^{n_r(d_N+1)}$ and $\gamma^*_n \in \R_+$ be the solutions of $\CPr$, and measurable in $\sigalg^n$. Then, the set \eqref{event} is $\sigalg^n$-measurable, and for every $\beta \in (0,1)$ and any $n$ such that
			\begin{align*}
				n \ge \frac{2}{\eps}\Big(\ln\frac{n_f(d_F + d_N + 1)}{\beta} + n_r(d_N + 1) + 1 \Big),
			\end{align*}
			where $d_N$ is the degree of the filter and $n_f, n_r, d_F$ are the system size parameters of \eqref{model}, we have
			\begin{align*}
				\PP^n\Big( \mathcal{E}\big(\bar N_n \opt, \gamma^*_n\big) \Big) < \beta.
			\end{align*} 
		\end{Thm}
		
	\begin{proof}
		See Appendix \ref{FDI:subsec:pf:perf}.
	\end{proof}
	
%==========================================================================================
\section{Cyber-Physical Security of Power Systems: AGC Case Study} \label{FDI:sec:sim}
%==========================================================================================
	In this section, we illustrate the performance of our theoretical results to detect a cyber intrusion in a two-area power system. Motivated by our earlier studies \cite{ref:Mohajerin_ACC10, ref:Mohajerin_CDC10}, we consider the IEEE 118-bus power network equipped with primary and secondary frequency control. While the primary frequency control is implemented locally, the secondary loop, referred also as AGC (Automatic Generation Control), is closed over the SCADA system without human operator intervention. As investigated in \cite{ref:Mohajerin_ACC10}, a cyber intrusion in this feedback loop may cause unacceptable frequency deviations and potentially load shedding or generation tripping. If the intrusion is, however, detected on time, one may prevent further damage by disconnecting the AGC. We show how to deploy the methodology developed in earlier sections to construct an FDI filter that uses the available measurements to diagnose an AGC intrusion sufficiently fast, despite the presence of unknown load deviations. 
	%------------------------------------------------------------------------------------------
	\subsection{Mathematical Model Description}
	%------------------------------------------------------------------------------------------
	In this section a multi-machine power system, based only on frequency dynamics, is described \cite{ref:Goran_mod}. The system is arbitrarily divided into two control areas. The generators are equipped with primary frequency control and each area is under AGC which adjusts the generating setpoints of specific generators so as to regulate frequency and maintain the power exchange between the two areas to its scheduled value.
	
	\subsubsection{System description}

	We consider a system comprising $n$ buses and $g$ number of generators. Let $G = \{i\}_1^g$ denote the set of generator indices and $A_1 = \{i \in G ~|~ i \text{ in Area 1} \}$, $A_2 = \{i \in G ~|~ i \text{ in Area 2} \}$ the sets of generators that belong to Area 1 and Area 2, respectively. Let also $$L_{tie}^k = \{(i,j)| i,j \text{~edges of a tie line from area} ~k \text{ to the other areas}\},$$ 
	where a tie line is a line connecting the two independently controlled areas and let also $K=\{1,2\}$ be the set of the indices of the control areas in the system.
	 
	 Using the classical generator model every synchronous machine is modeled as constant voltage source behind its transient reactance. The dynamic states of the system are the rotor angle $\delta_i$ $(rad)$, the rotor electrical frequency $f_i$ $(Hz)$ and the mechanical power (output of the turbine) $P_{mi}$ $(MW)$ for each generator $i\in G$. We also have one more state that represents the output of the AGC $\Delta P_{agc_k}$ for each control area $k\in K$.
	 % This state is introduced because the AGC is a PI controller.
	 
	 We denote by $E_G \in\mathbb{C}^{g}$ a vector consisting of the generator internal node voltages $E_{Gi}=|E_{Gi}^0|\angle{\delta_i}$ for $i\in G$. The phase angle of the generator voltage node is assumed to coincide with the rotor angle $\delta_i$ and $|E_{Gi}^0|$ is a constant. The voltages of the rest of the nodes are included in $V_N \in\mathbb{C}^{n}$, whose entries are $V_{Ni}=|V_{Ni}|\angle{\theta_i}$ for $i=1,\ldots,n$. To remove the algebraic constraints that appear due to the Kirchhoff's first law for each node, we retain the internal nodes (behind the transient reactance) of the generators and eliminate the rest of the nodes. This could be achieved only under the assumption of constant impedance loads since in that way they can be included in the network admittance matrix. The node voltages can then be linearly connected to the internal node voltages, and hence to the dynamic state $\delta_i$. This results in a reduced admittance matrix that corresponds only to the internal nodes of the generators, where the power flows are expressed directly in terms of the dynamic states of the system. The resulting model of the two area power system is described by the following set of equations.
		 \begin{align*}
			 \dot{\delta}_i &= 2 \pi (f_i-f_0),  \\ 
			 \dot{f}_i &= \frac{f_0}{2 H_i S_{B_i}}(P_{m_i}-P_{e_i}(\delta) - \frac{1}{D_i}(f_i-f_0)-\Delta P_{load_i}), \\
			 \dot{P}_{m,{a_k}} &= \frac{1}{T_{ch,a_k}}(P_{m,{a_k}}^{0} + v_{a_k} \Delta P_{p,{a_k}}^{sat} + w_{a_k}\Delta P_{agc,k}^{sat} - P_{m,{a_k}}), \\
			 \Delta \dot{P}_{agc,k} &=  \sum_{j \in A_k}c_{kj}(f_j-f_0) + \sum_{j \in A_k} b_{kj}(P_{m_j}-P_{e_j}(\delta)-\Delta P_{load_j}) \\
			 & \qquad \qquad - \frac{1}{T_{N_k}} g_k(\delta,f) - C_{p_k}h_k(\delta,f) -\frac{K_{k}}{T_{N_k}}(\Delta P_{agc,k}-\Delta P_{agc,k}^{sat}).
		 \end{align*}
	where $i\in G$, $a_k\in A_k$ for $k\in K$. Supperscript $sat$ on the AGC output signal $\Delta P_{agc,k}$ and on the primary frequency control signal $\Delta P_{p,{a_k}}$ highlights the saturation to which the signals are subjected. The primary frequency control is given by {$\Delta P_{p,i}=-(f_i-f_0)/S_i$}. Based on the reduced admittance matrix, the generator electric power output is given by  
	$$P_{ei} =  \sum_{j=1}^{g} {E_{G_i} E_{G_j} (  G_{ij}^{red} \cos (\delta_i - \delta_j) + B_{ij}^{red} \sin(\delta_i - \delta_j) ).}$$
	Moreover, {$g_k  \Let  \sum_{(i,j) \in L_{tie}^k}(P_{ij}-P_{T_{12}^0})$ } and {$h_k \Let {\diff g_k \over \diff t}$}, where the power flow $P_{ij}$, based on the initial admittance matrix of the system, is given by
	 \begin{align*}
	 	&P_{ij} = |V_{N_i}|| V_{N_j}| ( G_{ij} \cos(\theta_i-\theta_j) + B_{ij} \sin(\theta_i-\theta_j) )
	 \end{align*}
	 All undefined variables are constants, and details on the derivation of the models can be found in \cite{ref:Mohajerin_CDC12}. The AGC attack is modeled as an additive signal to the AGC signal. For instance, if the attack signal is imposed in Area 1, the mechanical power dynamics of Area 1 will be modified as
		\begin{align*}
			 &\dot{P}_{m,{a_1}} = \frac{1}{T_{ch,a_1}}(P_{m,{a_1}}^{0} + v_{a_1}\Delta P_{p,{a_1}}^{sat} + w_{a_1}\big(\Delta P_{agc_1}^{sat}+ f(t) \big)- P_{m,{a_1}}),%
		 \end{align*}

	 The above model can be compactly written as
	\begin{align}
	\label{ode-area}
		\begin{cases}
			\dot{X}(t) = h(X(t)) + B_d d(t) + B_f f(t)\\
			Y(t) = CX(t),
		\end{cases}
	\end{align}
	where $ X \Let \big[\{\delta_i\}_{1:g}, \{f_i\}_{1:g}, \{P_{m,i}\}_{1:g}, \{\Delta P_{{agc}_i}\}_{1:2}\big ]\tr \in \R^{3g+2}$ denotes the internal states vector comprising rotor angles $\delta_i$, generators frequencies $f_i$, generated mechanical powers $P_{m,i}$, and the AGC control signal $\Delta P_{{agc}_i}$ for each area. The external input $d \Let \big[\{\Delta P_{load_i}\}_{1:g}\big]\tr$ represents the unknown load disturbances (discussed in the next subsection), and $f$ represents the intrusion signal injected to the AGC of the first area. We assume that the measurements of all the frequencies and generated mechanical power are available, i.e., $Y = \big[\{f_i\}_{1:g}, \{P_{m,i}\}_{1:g}]\tr \in \R^{2g}$. The nonlinear function $h(\cdot)$ and the constant matrices $B_d$, $B_f $ and $C$ can be easily obtained by the mapping between the analytical model and (\ref{ode-area}). To transfer the ODE dynamic expression \eqref{ode-area} into the DAE \eqref{model} it suffices to introduce %$ x \Let [X\tr - X_e\tr, d\tr]\tr$ , $ z \Let Y - C X_e $ and
	\begin{small}
		\begin{align*}
			x &\Let \begin{bmatrix} X - X_e \\ d \end{bmatrix}, 
			\qquad z \Let Y - C X_e \\
			E(x) \Let \begin{bmatrix}
			h(X) - A(X-X_e)\\
				0
				\end{bmatrix}&, 
			\quad H(p) \Let \begin{bmatrix}
		                  -pI + A & B_d \\
		                   C & 0
		               \end{bmatrix},
			\quad L(p) \Let \begin{bmatrix}
		                  0  \\
		                  -I 
		               \end{bmatrix},
			\quad F(p) \Let \begin{bmatrix}
		                  B_f \\
		                  0 
		               \end{bmatrix},
		\end{align*}
	\end{small}
	where $X_e$ is the equilibrium of \eqref{ode-area}, i.e., $h(X_e) = 0$, and $A \Let \frac{\partial h}{\partial X} \big |_{X = X_e}$. Notice that by the above definition, the nonlinear term $E(\cdot)$ only carries the nonlinearity of the system while the linear terms of the dynamic are incorporated into the constant matrices $H, L, F$. This can always be done without loss of generality, and practically may improve the performance of the scheme, as the linear terms can be fully decoupled from the residual of the filter.

	\subsubsection{Load Deviations and Disturbances}
	\begin{figure}[t!]
		\centering
		\includegraphics[width=1\columnwidth]{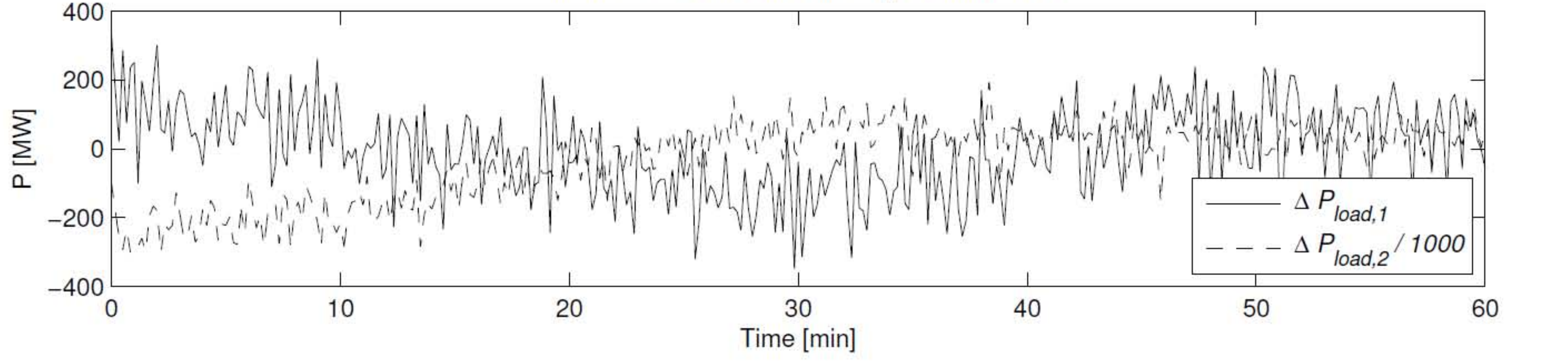}
		\caption{Stochastic load fluctuation and prediction error \cite[p.\ 59]{ref:Goran_dyn}}\label{fig:load}
	\end{figure}
	
	Small power imbalances arise during normal operation of power networks due, for example, to load fluctuation, load forecast errors, and trading on electricity market. Each of these sources give rise to deviations at different time scale. High frequency load fluctuation is typically time uncorrelated stochastic noise on a second or minute time scale, whereas forecast errors usually stem from the mismatch of predicted and actual consumption on a 15-minute time scale. Figure \ref{fig:load} demonstrates two samples of stochastic load fluctuation and forecast error which may appear at two different nodes of the network \cite[p.\ 59]{ref:Goran_dyn}. The trading on the electricity market also introduces disturbances, for example, in an hourly framework (depending on the market). 
	
	To capture these sources of uncertainty we consider a space of disturbance patterns comprising combinations of sinusoids at different frequency ranges (to model short term load fluctuation and mid-term forecast errors) and step functions (to model long-term abrupt changes due to the market).
	%Despite the inherent stochasticity of these disturbances, one may exploit these information to model the behavior of the load deviations. For example, we can describe the disturbances of the first two categories in Figure \ref{fig:load} via a family of sinusoidal functions concentrated on different frequency regions, i.e., the high frequency modes correspond to the slower modes concern the prediction mismatch. In regard to the electricity market, the hourly abrupt changes in power imbalances can be captured by a step function with a random power. In this light, 
	The space of load deviations (i.e., the disturbance patterns $\D$ in our FDI setting) is then described by 
	\begin{align}
	\label{load}
		\Delta P_{load}(t) \Let \alpha_0 + \sum_{i=1}^\eta \alpha_i \sin(\omega_i t + \phi_i), \qquad t \in [0,T],
	\end{align}
	where the parameters $(\alpha_i)^\eta_{i=0}$, $(\omega_i)^\eta_{i=1}$ $(\phi_i)^\eta_{i=1}$, and $\eta$ are random variables whose distributions induce the probability measure on $\D$. We assume that $\sum_{i=0}^{\eta} |\alpha_i|^2$ is uniformly bounded with probability 1 to meet the requirements of Theorem \ref{thm:AP}.

	%------------------------------------------------------------------------------------------
	\subsection{Diagnosis Filter Design}\label{FDI:subsec:sim:filter}
	%------------------------------------------------------------------------------------------

	To design the FDI filter, we set the degree of the filter $d_N = 7$, the denominator $a(p) = (p+2)^{d_N}$, and the finite time horizon $T = 10 \sec$. Note that the degree of the filter is significantly less than the dimension of the system \eqref{ode-area}, which is $59$. This is a general advantage of the residual generator approach in comparison to the observer-based approach where the filter order is effectively the same as the system dynamics. To compute the signature matrix $Q_x$, we resort to the finite dimensional approximation $Q_{\B}$ in Proposition \ref{prop:app}. Inspired by the class of disturbances in \eqref{load}, we first choose Fourier basis with $80$ harmonics
	\begin{align}
		\label{basis}
		b_{i}(t) \Let \begin{cases} \cos(\frac{i}{2}\omega t) \quad &i~\text{: even}\\ \sin(\frac{i+1}{2}\omega t) \quad  &i~\text{: odd} \quad \end{cases}, \qquad \qquad  \omega \Let \frac{2\pi}{T}, \quad i \in \{0,1,\cdots, 80\}.
	\end{align} 
	We should emphasize that there is no restriction on the basis selection as long as Assumptions \ref{a:basis} are fulfilled; we refer to \cite[Section V.B]{ref:Mohajerin_CDC12} for another example with a polynomial basis. Given the basis \eqref{basis}, it is easy to see that the differentiation matrix $D$ introduced in \eqref{D} is
		\begin{align*}
		   & D = \begin{bmatrix}
				0& 0			& 0		&		& \cdots	& 0	& 0 \\
				0& 0			&\omega	&	& \cdots	& 0	& 0 \\
				0& -\omega	& 	0	&	& \cdots 	& 0	& 0 \\
				\vdots&\vdots&\vdots&	&\ddots	&\vdots&\vdots \\
				0&	0		& 	0	&	&  	& 0	& 80\omega\\
				0& 0 		& 	0&	& \cdots& -80\omega& 0
				\end{bmatrix}.
		\end{align*}	
	We can also compute offline (independent of $x$) the matrix $G$ in \eqref{G} with the help of the basis \eqref{basis} and the denominator $a(p)$. To proceed with $Q_x$ of a sample $\Delta P_{load}$ we need to run the system dynamic \eqref{ode-area} with the input $d(\cdot) \Let \Delta P_{load}$ and compute $x(t) \Let [ X(t)\tr, \Delta P_{load}(t)]\tr$ where $X$ is the internal states of the system. Given the signal $x$, we then project the nonlinearity signature $t \mapsto e_x(t) \teL E\big(x(t)\big)$ onto the subspace $\B$ (i.e., $\proj(e_x)$), and finally obtain $Q_x$ from \eqref{Q}. In the following simulations, we deploy the YALMIP toolbox \cite{ref:YALMIP} to solve the corresponding optimization problems. 

		\begin{figure}[t!]
			\centering
			\includegraphics[width = 0.7\columnwidth]{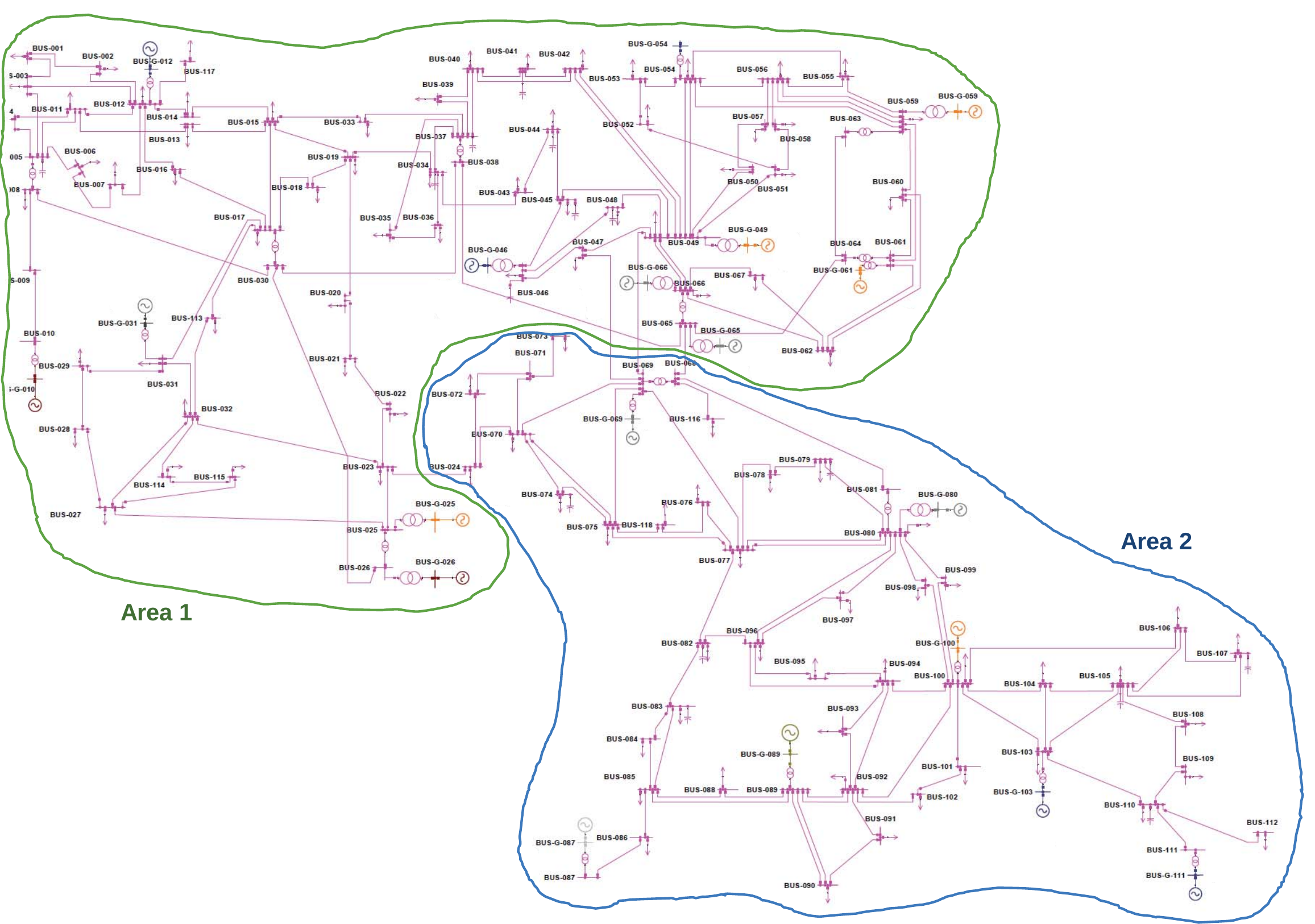}
			\caption{IEEE 118-bus system divided into two control areas}\label{fig:2area}
		\end{figure}

	%------------------------------------------------------------------------------------------
	\subsection{Simulation Results} \label{FDI:subsec:test}
	%------------------------------------------------------------------------------------------
		\subsubsection{Test system}

	To illustrate the FDI methodology we employed the IEEE 118-bus system. The data of the model are retrieved from a snapshot available at \cite{ref:data118}. It includes 19 generators, 177 lines, 99 load buses and 7 transmission level transformers. Since there were no dynamic data available, typical values provided by \cite{ref:Andreson_Fouad} were used for the simulations. The network was arbitrarily divided into two control areas whose nonlinear frequency model was developed in the preceding subsections. Figure \ref{fig:2area} depicts a single-line diagram of the network and the boundaries of the two controlled areas where the first and second area contain, respectively, $12$ and $7$ generators.
	
		\subsubsection{Numerical results}
	In the first simulation we consider the scenario that an attacker manipulates the AGC signal of the first area at $T_{ack} = 10 \sec$. We model this intrusion as a step signal equal to $14~MW$ injected into the AGC in Area 1. To challenge the filter, we also assume that a step load deviation occurs at $T_{load} = 1 \sec$ at node $5$. In the following we present the results of two filters: Figure \ref{fig:LP} shows the filter based on formulation \eqref{LP-opt} in Approach (I), which basically neglects the nonlinear term; Figure \ref{fig:QP} shows the proposed filter in \eqref{APr} based on $\Ap$ perspective where the payoff function is $J(\alpha) \Let \alpha^2$; see Remark \ref{rem:AP} why such a payoff function is of particular interest. 
	
	\begin{figure}[t]
		\centering
			\subfigure[Performance of the filter neglecting the nonlinear term]{\label{fig:LP}\includegraphics[scale = 0.14]{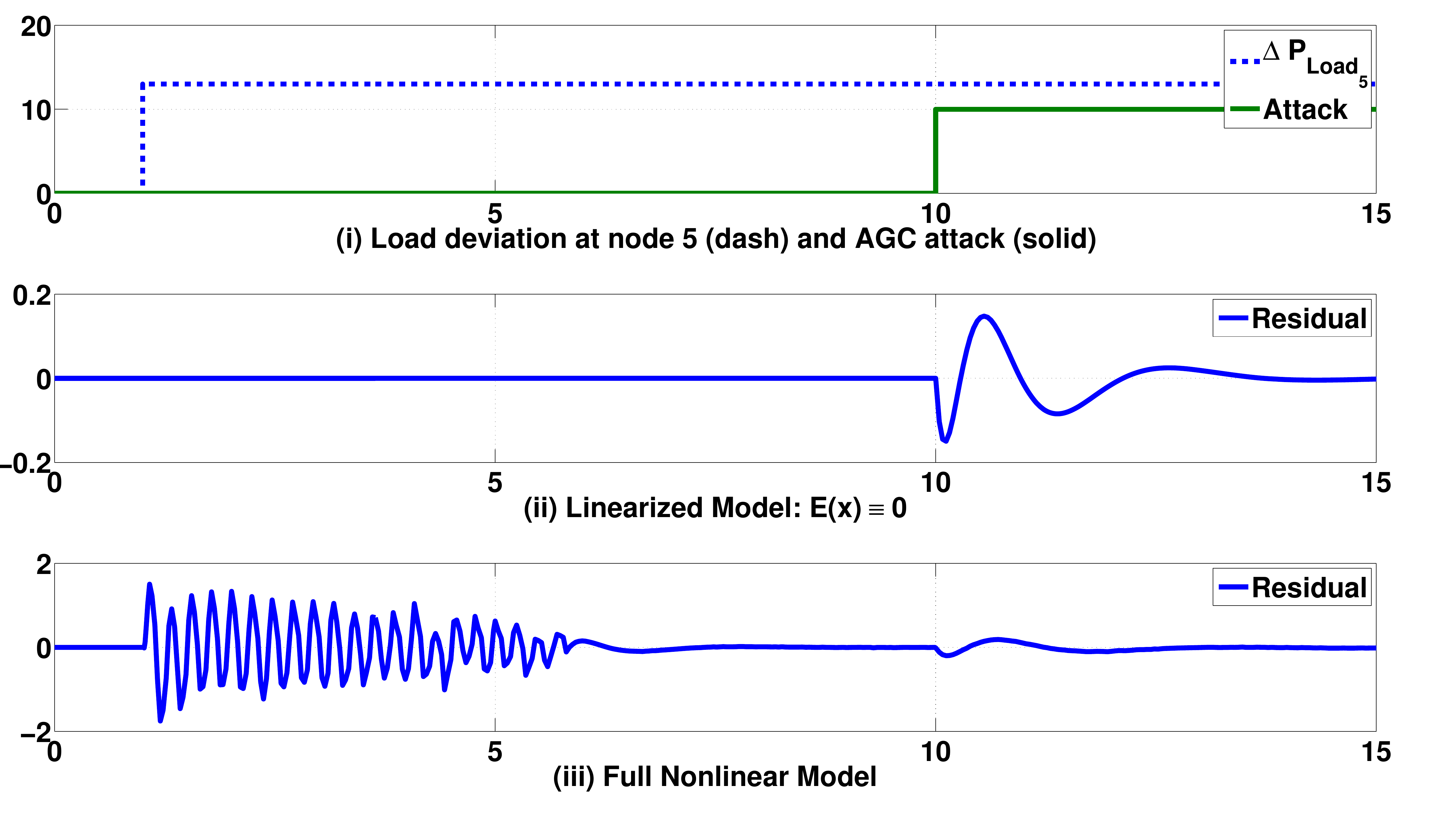}}  
			\subfigure[Performance of the filter trained for the step signatures]{\label{fig:QP}\includegraphics[scale = 0.14]{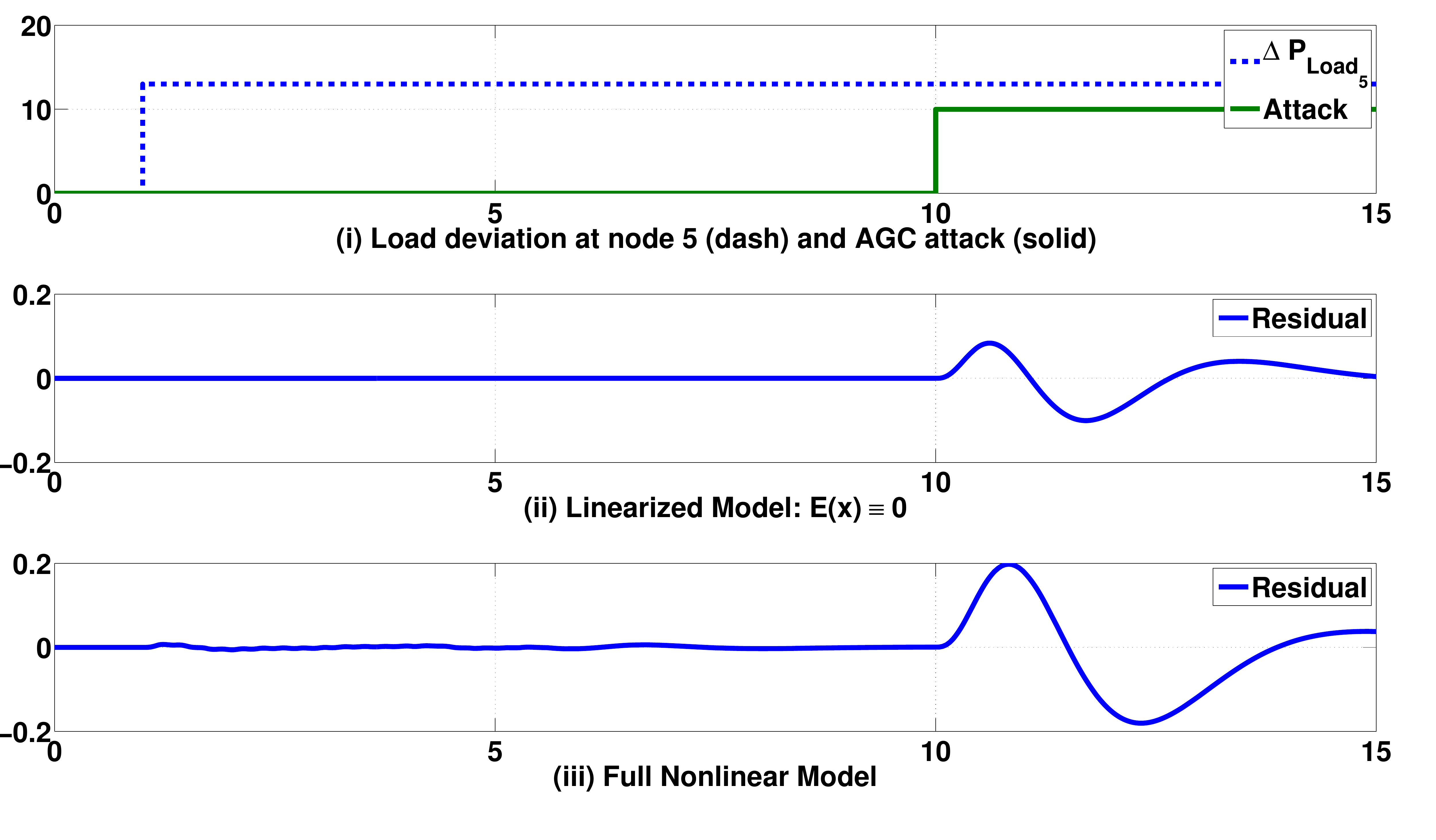}}
			\caption{Performance of the FDI filters with step inputs}
			\label{fig:step}
	\end{figure}

	We validate the filters performance with two sets of measurements: first the measurements obtained from the linearized dynamic (i.e.\ $E(x) \equiv 0$); second the measurements obtained from the full nonlinear model \eqref{ode-area}. As shown in Fig.\ \ref{fig:LP}(ii) and Fig.\ \ref{fig:QP}(ii), both filters work perfectly well with linear dynamics measurements. It even appears that the first filter seems more sensitive. However, Fig.\ \ref{fig:LP}(iii) and Fig.\ \ref{fig:QP}(iii) demonstrate that in the nonlinear setting the first filter fails whereas the robustified filter works effectively similar to the linear setting.  
	
	In the second simulation, to evaluate the filter performance in more realistic setup, we robustify the filter to random disturbance patterns, and then verify it with new generated samples. To measure the performance in the presence of the attack, we introduce the following indicator:
		\begin{align}
		\label{rho}
			\rho \Let \frac{\max\limits_{t \le T_{ack}} \|r(t)\|_{\infty}}{\max\limits_{t \le T}\|r(t)\|_{\infty}},
		\end{align}
	where $r$ is the residual \eqref{non-res}, and $T_{ack}$ is when the attack starts. Observe that $\rho \in [0,1]$, and the lower $\rho$ the better performance for the filter, e.g., in Fig.\ \ref{fig:LP}(iii) $\rho = 1$, and in Fig.\ \ref{fig:QP}(iii) $\rho \approx 0$. 

	\begin{figure}[t]
		\includegraphics[width = 0.6\textwidth]{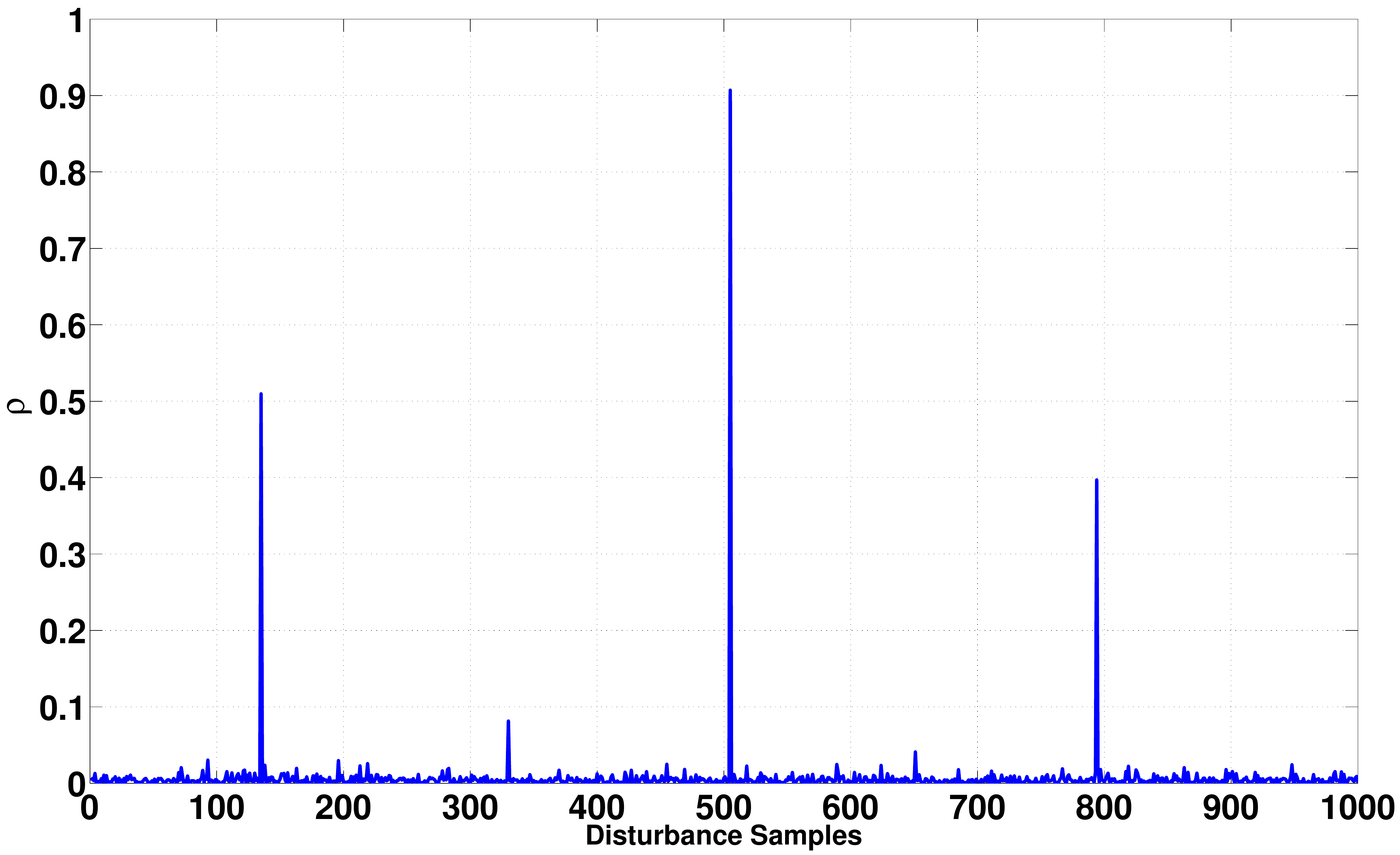}
		\centering
		\caption{The indicator $\rho$ defined in \eqref{rho}}
		\label{fig:rho}
	\end{figure}
	
	In the training phase, we randomly generate five sinusoidal load deviations as described in \eqref{load}, and excite the dynamics for $T = 10~sec$ in the presence of each of the load deviations individually. Hence, in total we have $n = 19 \times 5 = 95$ disturbance signatures. Then, we compute the filter coefficients by virtue of $\APr$ in \eqref{APr} with the payoff function $J(\alpha) \Let \alpha^2$ and these $95$ samples. In the operation phase, we generate two new disturbance patterns with the same distribution as in the training phase and run the system in the presence of both load deviations simultaneously at two random nodes for the horizon $T = 120 \sec$. Meanwhile, we inject an attack signal at $T_{ack} = 110 \sec$ in the AGC, and compute the indicator $\rho$ in \eqref{rho}. Figure \ref{fig:rho} demonstrates the result of this simulation for $1000$ experiments. 
%==========================================================================================
\section{Conclusion and Future Directions} \label{FDI:sec:sum}
%==========================================================================================
	In this article, we proposed a novel perspective toward the FDI filter design, which is tackled via an optimization-based methodology along with probabilistic performance guarantees. Thanks to the convex formulation, the methodology is applicable to high dimensional nonlinear systems in which some statistical information of exogenous disturbances are available. Motivated by our earlier works, we deployed the proposed technique to design a diagnosis filter to detect the AGC malfunction in two-area power network. The simulation results validated the filter performance, particularly when the disturbance patterns are different from training to the operation phase. 
	
	The central focus of the work here is to robustify the filter to certain signatures of dynamic nonlinearities in the presence of given disturbance patterns. As a next step, motivated by applications that the disruptive attack may follow certain patterns, a natural question is whether the filter can be trained to these attack patterns. From the technical standpoint, this problem in principle may be different from the robustification process since the former may involve maximization of the residual norm as opposed to the minimization for the robustification discussed in this article. Therefore, this problem offers a challenge to reconcile the disturbance rejection and the fault sensitivity objectives. 
	
	The proposed methodology in this study is applicable to both discrete and continuous-time dynamics and measurements. In reality, however, we often have different time-setting in different parts, i.e., we only have discrete-time measurements while the system dynamics follows a continuous-time behavior. We believe this setup introduces new challenges to the field. We recently reported heuristic attempts toward this objective in \cite{ref:Eralia-CDC13}, though there is still a need to address this problem in a rigorous and systematic framework.

%==========================================================================================
\section*{Acknowledgment}
%==========================================================================================

The authors are grateful to M.\ Vrakopoulou and G.\ Andersson for the help on the AGC case study. The first author also thanks G.\ Schildbach for fruitful discussions on randomized algorithms.

%===============================================================================
\renewcommand{\thesection}{I}
\section{Appendix}\label{FDI:sec:app}
\setcounter{equation}{0}
\numberwithin{equation}{section}
%===============================================================================

%------------------------------------------------------------------------------------------
	\subsection{Proofs of Section \ref{FDI:subsec:nonlinear}} \label{FDI:subsec:pf:nonlinear}
%------------------------------------------------------------------------------------------
	Let us start with a preliminary required for the main proof of this section.

	\begin{Lem}
	\label{lem:Hnorm}
		Let $N(p) \Let \sum^{d_N}_{i=0} N_i p^i$ be an $\R^{n_r}$ row polynomial vector with degree $d_N$, and $a(p)$ be a stable polynomial with the degree at least $d_N$. Let $\bar{N} \Let [N_0 \ N_1 \ \cdots \ N_{d_N}]$ be the collection of  the coefficients of $N(p)$. Then, 
		\begin{align*}
			\big \|a^{-1}N \big\|_{\Hinf} \le \widetilde C \|\bar N\|_{\infty}, \qquad \widetilde C\Let \sqrt{n_r (d_N + 1)} ~ \|a^{-1}\|_{\Hinf}.
		\end{align*}
	\end{Lem}
	
	\begin{proof}
		Let $b(p) \Let \sum^{d_N}_{i=0} b_i p^i$ be a polynomial scaler function. By $\Hinf$-norm definition we have
		\begin{align}
		\label{eq:1}
			\big \|a^{-1}b \big\|_{\Hinf}^2 = \sup_{\omega \in (-\infty,\infty)} \Big|  \frac{b(j\omega)}{a(j\omega)} \Big|^2 \le \sup_{\omega \in [0,\infty)} \frac{ \sum^{d_N}_{i=0}|b_i|^2 \omega^{2i}}{|a(j\omega)|^2}. 
		\end{align}
		Let $\bar b \Let \begin{bmatrix} b_0 & b_1 & \cdots & b_{d_N} \\ \end{bmatrix}$. It is then straightforward to inspect that 
		\begin{align}
		\label{eq:2}
			\sum^{d_N}_{i=0}|b_i|^2 \omega^{2i} \le 
			\begin{cases}
				(d_N + 1) \| \bar b \|_{\infty}^2 &\text{if} \quad \omega \in [0,1] \\
				(d_N + 1) \| \bar b \|_{\infty}^2 \omega^{2d_N} &\text{if} \quad \omega \in (1,\infty)
			\end{cases}
		\end{align}
		Therefore, \eqref{eq:1} together with \eqref{eq:2} yields to 
		\begin{align*}
			\big \|a^{-1}b \big\|_{\Hinf}^2 \le (d_N + 1) \big \|a^{-1} \big \|^2_{\Hinf} \|\bar{b}\|_\infty^2.
		\end{align*}
		Now, taking the dimension of the vector $N(p)$ into consideration, we conclude the desired assertion.		
	\end{proof}

	\begin{proof}[Proof of Lemma \ref{lem:psi}]
		Let $\ell \ge d_N$ be the degree of the scalar polynomial $a(p)$. Then, taking advantage of the state-space representation of the matrix transfer function $a^{-1}(p)N(p)$, in particular the observable canonical form \cite[Section 3.5]{ref:ZhouDoyle}, we have
			\begin{align*}
				r_x(t) = \int_0^t C\e^{-A(t-\tau)}B e_x(\tau) \diff \tau + De_x(t), 
			\end{align*}
		where $C \in \R^{1 \times \ell}$ is a constant vector, $A \in \R^{\ell \times \ell}$ is the state matrix depending only on $a(p)$, and $B \in \R^{\ell \times n_r}$ and $D \in \R^{1 \times n_r}$ are matrices that depend linearly on all the coefficients of the numerator $\bar N \in \R^{n_r(d_N+1)}$. Therefore, it can be readily deduced that \eqref{psi-rx} holds for some function $\psi_x \in \W^{n_r(d_N+1)}_T$. In regard to \eqref{psi-rx} and the definition \eqref{r_x}, we have
			\begin{align}
			\label{eq:app1}
				\| \bar N \psi_x \|_{\Lnorm} = \|r_x\|_{\Lnorm} =  \big\| a^{-1}(p)N(p)e_x \big\|_{\Lnorm} \le \big\|a^{-1}N \big\|_{\Hinf}\|e_x\|_{\Lnorm} \le \widetilde C \|\bar N\|_{\infty} \|e_x\|_{\Lnorm},
			\end{align}
		where the first inequality follows from the classical result that the $\Lnorm$-gain of a matrix transfer function is the $\Hinf$-norm of the matrix \cite[Theorem 4.3, p.\ 51]{ref:ZhouDoyle}, and the second inequality follows from Lemma \ref{lem:Hnorm}. Since \eqref{eq:app1} holds for every $\bar N \in \R^{n_r(d_N+1)}$, then 
		%Setting $\bar N$ as a vector with all zeros but one in \eqref{eq:app1}, one can infer that
			\begin{align*}
				\|\psi_x\|_{\Lnorm} \le \sqrt{n_r (d_N + 1)} ~ \widetilde C \|e_x\|_{\Lnorm},
			\end{align*}
		which implies \eqref{psi-norm}. 		
	\end{proof}

	\begin{proof}[Proof of Proposition \ref{prop:app}]
			Observe that by virtue of the triangle inequality and linearity of the projection mapping we have
				\begin{align*}
				%\label{eq:err-bound}
					\big | \|r_x \|_{\Lnorm} - \big\| a^{-1}(p)N(p)\proj(e_x)\big\|_{\Lnorm} \big | 
					%\le \big \|r_x - a^{-1}(p)N(p)\proj(e_x) \big \|_{\Lnorm}
					\le \big\| a^{-1}(p)N(p)\big(e_x - \proj(e_x)\big) \big\|_{\Lnorm} 
					%\le \big\|a^{-1} N \big\|_{\Hinf} \big\| e_x - \proj(e_x) \big\|_{\Lnorm} 
					\le \widetilde C \|\bar N \|_\infty \delta,
				\end{align*}
			where the second inequality follows in the same spirit as \eqref{eq:app1} and $\| e_x - \proj(e_x)\|_{\Lnorm} \le \delta$. Note that by definitions of $Q_x$ and $Q_{\B}$ in \eqref{Q_x} and \eqref{Q}, respectively, we have
				\begin{align*}
					\big | \bar N (Q_x - Q_{\B}) \bar N\tr \big | & = \big | \|r_x \|^2_{\Lnorm} - \big\| a^{-1}(p)N(p)\proj(e_x)\big\|^2_{\Lnorm} \big | \le \widetilde C \|\bar N \|_\infty \delta \big( \widetilde C \|\bar N \|_\infty \delta + 2\|r_x\|_{\Lnorm} \big )\\
					& \le 	\widetilde C^2 \|\bar N \|^2_\infty \delta \big( \delta	+ 2 \|e_x\|_{\Lnorm}\big) \le
					C \|a^{-1}\|_{\Hinf} \|\bar N \|^2_2\delta \big( 1 + 2 \|e_x\|_{\Lnorm}\big) 
				\end{align*}
				where the inequality of the first line stems from the simple inequality {\small $|\alpha^2 - \beta^2| \le |\alpha - \beta|\big(2|\alpha| + |\alpha - \beta|\big)$}, and $C$ is the constant as in \eqref{psi-norm}. 
	\end{proof}

%------------------------------------------------------------------------------------------
	\subsection{Proofs of Section \ref{FDI:subsec:perf}} \label{FDI:subsec:pf:perf}
%------------------------------------------------------------------------------------------
	To prove Theorem \ref{thm:AP} we need a preparatory result addressing the continuity of the mapping $\phi$ in \eqref{phi}.
		\begin{Lem}
			\label{lem:phi}
			Consider the function $\phi$ as defined in \eqref{phi}. Then, there exists a constant $L > 0$ such that for any $\bar N_1, \bar N_2 \in \NN$ and $x_1, x_2 \in \W^{n_x}_T$ where $\| x_i \|_{\Lnorm} \le M$, we have
				\begin{align*}
					\big| \phi(\bar N_1,x_1) - \phi(\bar N_2,x_2) \big| \le L \big( \big \| \bar N_1 - \bar{N}_2 \big\|_\infty + \|x_1 - x_2 \|_{\Lnorm} \big).
				\end{align*}
		\end{Lem}
		
	\begin{proof}
	Let $L_E$ be the Lipschitz continuity constant of the mapping $E:\R^{n_x} \ra \R^{n_r}$ in \eqref{r_x}. We modify the notation of $r_x$ in \eqref{r_x} with a new argument as $r_{x}[\bar N]$, in which $\bar N$ represents the filter coefficients. Then, with the aid of \eqref{eq:app1}, we have 
		\begin{align*}
			\sup_{ \|x\|_{\Lnorm} \le M} \sup_{\bar N \in \NN}  \|r_x[\bar N]\|_{\Lnorm} 
			\le \sup_{ \|x\|_{\Lnorm} \le M} \sup_{\bar N \in \NN} \widetilde C L_E \|\bar N\|_{\infty} \|x\|_{\Lnorm} 
			\le \widetilde M, \qquad \widetilde M \Let \widetilde{C} L_E M,
		\end{align*}
	where the constant $\widetilde{C}$ is introduced in Lemma \ref{lem:Hnorm}. As the payoff function $J$ is convex, it is then Lipschitz continuous over the compact set $[0,\widetilde M]$ \cite[Proposition 5.4.2, p.\ 185]{ref:Bert-09}; we denote this Lipschitz constant by $L_J$. Then for any $\bar N_i \in \NN$ and $\|x_i\|_{\Lnorm} \le M$, $i \in \{1,2\}$, we have, 
		\begin{align}
			\big|\phi(\bar N_1,x_1)- \phi(\bar N_2,x_2)\big| 
			\notag &\le L_J \big| \big\|r_{x_1}[N_1] \big\|_{\Lnorm} - \big\| r_{x_2}[N_2] \big\|_{\Lnorm} \big| \\
			\notag &\le L_J \Big( \big\|r_{x_1}[N_1] - r_{x_1}[N_2] \big\|_{\Lnorm}+  \big\| r_{x_1}[N_2] - r_{x_2}[N_2] \big\|_{\Lnorm} \Big)\\
			\label{eq:lem:1} &\le L_J \Big( \widetilde{C}\|e_{x_1}\|_{\Lnorm} \|N_1 - N_2\|_\infty +  \widetilde{C}\|e_{x_1} - e_{x_2}\|_{\Lnorm} \|N_2\|_\infty \Big)\\
			\notag & \le L_J\widetilde{C}L_E\big( M\|N_1 - N_2\|_\infty + \|{x_1} - {x_2}\|_{\Lnorm} \big). 			
		\end{align}	
	where \eqref{eq:lem:1} follows from \eqref{eq:app1} and the fact that the mapping $(\bar N, e_x) \mapsto r_x[\bar N]$ is bilinear.
	\end{proof}
		
		\begin{proof}[Proof of Theorem \ref{thm:AP}]
			By virtue of Lemma \ref{lem:phi}, one can infer that for every $\bar N \in \NN$ the mapping $x \mapsto \phi(\bar N, x)$ is continuous, and hence measurable. Therefore, $\phi(\bar N, x)$ can be viewed as a random variable for each $\bar N \in \NN$, which yields to the first assertion, see \cite[Chapter 2, p.\ 84]{ref:Bil-99} for more details. 
			
			By uniform (almost sure) boundedness and again Lemma \ref{lem:phi}, the mapping $\bar N \mapsto \phi(\bar N,x)$ is uniformly Lipschitz continuous (except on a negligible set), and consequently first moment continuous in the sense of \cite[Definition 2.5]{ref:Hansen-2012}. We then reach \ref{slln} by invoking \cite[Theorem 2.1]{ref:Hansen-2012}. 
			
			For assertion \ref{uclt}, note that the compact set $\NN$ is finite dimensional, and thus admits a logarithmic $\eps$-capacity in the sense of \cite[Section.\ 1.2, p.\ 11]{ref:Dudley99}. Therefore, the condition \cite[(6.3.4), p.\ 209]{ref:Dudley99} is satisfied. Since the other requirements of \cite[Theorem 6.3.3, p.\ 208]{ref:Dudley99} are readily fulfilled by the uniform boundedness assumption and Lemma \ref{lem:phi}, we arrive at the desired UCLT assertion in \ref{uclt}.						
		\end{proof}

	To keep the paper self-contained, we provide a proof for Theorem \ref{thm:CP} in the following, but refer the interested reader to \cite[Theorem 4.1]{ref:MohSut-13} for a result of a more general setting.
	
	\begin{proof}[Proof of Theorem \ref{thm:CP}]
		The measurability of $\mathcal{E}$ is a straightforward consequence of the measurability of $[\bar N_n \opt, \gamma_n \opt]$ and Fubini's Theorem \cite[Theorem 18.3, p.\ 234]{ref:Bil-95}. For notational simplicity, we introduce the following notation. Let $\ell \Let n_r(d_N + 1)+1$ and define the function $f: \R^{\ell} \times \W^{n_x}_T \ra \R$ 
				\begin{align*}
					f(\theta, x) \Let \bar N Q_x \bar N\tr - \gamma, \qquad \theta \Let [\bar N, \gamma ]\tr \in \R^\ell, 
				\end{align*}	
		where $Q_x$ is the nonlinearity signature matrix of $x$ as defined in \eqref{Q_x}, and $\theta$ is the augmented vector collecting all the decision variables. Consider the convex sets $\Theta_j \subset \R^\ell$
			\begin{align*}
				%\label{theta:FDI}
				\Theta_j \Let \Big\{ \theta = [\bar N, \gamma]\tr ~ \big | ~ \bar N  \bar H = 0, ~ \bar N \bar F v_j \ge 1\Big \}, 
				\qquad v_j \Let \overset{ \quad \da ~j^{\text{th}} }{\big[0,\cdots, {1} , \cdots, 0 \big]\tr}, %\begin{bmatrix} 0 & \cdots & 1 & \cdots& 0 \end{bmatrix}\tr,
			\end{align*}
				%\underset{\begin{array}{c} \uparrow \\ j^{\text{th}} \end{array}}{1}
		where the size of $v_j$ is $m \Let n_f(d_F + d_N + 1)$. Note that in view of Lemma \ref{lem:inf norm}, we can replace the characterization of the filter coefficients in \eqref{LP} with $\theta \in \bigcup_{j = 1}^m \Theta_j$. We then express the program $\Cp$ in \eqref{rcp-ccp} and its random counterpart $\CPr_1$ in \eqref{CPr_1} as follows:
			\begin{align*} 
			%\label{CC}
				\Cp: \left\{ \begin{array}{cll}
						\min\limits_{\theta \in \bigcup\limits_{j = 1}^m \Theta_j}		&  c\tr \theta   \\
						 \text{s.t. }	& \PP \big( f(\theta,x) \le 0 \big) \ge 1- \eps 
							\end{array} \right.
				\quad  
				\CPr_1: \left\{ \begin{array}{cll}
						\min\limits_{\theta \in \bigcup\limits_{j = 1}^m \Theta_j}		&  c\tr \theta   \\
						 \text{s.t. }	& \max\limits_{i \le n}f(\theta,x_i) \le 0,
							\end{array} \right.
			\end{align*}
		where $c$ is the constant vector with $0$ elements except the last which is $1$. It is straightforward to observe that the optimal threshold $\gamma_n\opt$ of the two-stage program $\CPr$ in \eqref{CPr} is the same as the optimal threshold obtained in the first stage $\CPr_1$. Thus, it suffices to show the desired assertion considering only the first stage. Let $\theta_n\opt \Let [\bar N_n \opt, \gamma_n \opt]$ denote the optimizer of $\CPr_1$. Now, consider $m$ sub-programs denoted by $\CP{(j)}$ and $\CPr{(j)}$ for $j \in \{1,\cdots,m\}$:
			\begin{align*} 
			%\label{CC}
				\Cp{(j)}: \left\{ \begin{array}{cll}
						\min\limits_{\theta \in \Theta_j}		&  c\tr \theta   \\
						 \text{s.t. }	& \PP \big( f(\theta,x) \le 0 \big) \ge 1- \eps 
							\end{array} \right.
				\quad  
				\CPr{(j)}: \left\{ \begin{array}{cll}
						\min\limits_{\theta \in \Theta_j}		&  c\tr \theta   \\
						 \text{s.t. }	& \max\limits_{i \le n}f(\theta,x_i) \le 0,
							\end{array} \right.
			\end{align*}
		Let us denote the optimal solution of $\CPr{(j)}$ by $\theta^*_{n,j}$. Note that for all $j$, the set $\Theta_j$ is deterministic (not affected by $x$) and convex, and the corresponding random program $\CPr{(j)}$ is feasible if $\Theta_j \neq \emptyset$, thanks to the min-max structure of $\CPr{(j)}$. Therefore, we can readily employ the existing results of the random convex problems. Namely, by \cite[Theorem 1]{ref:CamGar-08} we have
			\begin{align*}
					\PP^n \big(\mathcal{E}(\theta^*_{n,j})\big)	< \sum_{i = 0}^{\ell - 1} {n \choose i} \eps^{i}(1-\eps)^{n-i}, \qquad \forall j \in \{1, \cdots, m\} 
					%\ell \Let n_r(d_N+1) + 2,
			\end{align*}
		where $\mathcal{E}$ is introduced in \eqref{event}. Furthermore, it is not hard to inspect that $\theta^*_{n} \in \big(\theta^*_{n,j}\big)_{j = 1}^{m}$. Thus, $\mathcal{E}(\theta^*_n) \subseteq \bigcup_{j=1}^m \mathcal{E}(\theta^*_{n,j})$ which yields 
		\begin{align*}
			\PP^n\big(\mathcal{E}(\theta^*_n) \big) \le \PP^n\Big( \bigcup_{j=1}^m \mathcal{E}(\theta^*_{n,j}) \Big) \le \sum_{j=1}^m \PP^n\big(\mathcal{E}(\theta^*_{n,j}) \big) < m \sum_{i = 0}^{\ell - 1} {n \choose i} \eps^{i}(1-\eps)^{n-i}.
		\end{align*}
		Now, considering $\beta$ as an upper bound, the desired assertion can be obtained by similar calculation as in \cite{ref:Cal-09} to make the above inequality explicit for $n$ in terms of $\eps$ and $\beta$. 		
	\end{proof}
	
%==========================================================================================
	\bibliographystyle{amsalpha} %abbrv siam %plain
	\bibliography{ref,ref_MohajerinEsfahani}

\end{document}